\documentclass[a4paper, 11pt]{article}
\usepackage{latexsym} 
\usepackage{amsfonts} 
\usepackage{amsmath} 
\usepackage{amsthm} 
\usepackage{mathrsfs} 
\usepackage{dsfont} 
\usepackage{bbold} 
\usepackage[english]{babel} 
\usepackage{caption} 
\usepackage{epsfig} 
\usepackage{float} 
\usepackage{subfigure} 
\usepackage{psfrag} 
\usepackage{graphicx} 
\usepackage{epsfig} 
\usepackage{amsbsy} 
\usepackage{hyperref}
\usepackage{xcolor}

\font\titlefont=cmr12 at 14pt
 
\DeclareMathOperator{\Val}{\matV}

\makeatletter
\newtheoremstyle{break}
   {\topsep}{\topsep}%
   {\itshape}{}%
   {\bfseries}{.}%
   { }
   {\thmname{#1}\thmnumber{\@ifnotempty{#1}{ }\@upn{#2}}%
    \thmnote{ {\bfseries(#3)}}}%
\makeatother
\theoremstyle{break}

\newtheorem{theorem}{Theorem} 
\newtheorem{thm}{Theorem} 
\newtheorem*{hyp*}{Hypothesis} 
\newtheorem{coro}[theorem]{Corollary} 
\newtheorem{defi}[theorem]{Definition} 
\newtheorem{lemma}[theorem]{Lemma} 
\newtheorem{prop}[theorem]{Proposition} 
\newtheorem{rmk}[theorem]{Remark}

\newtheorem{hyp}{Hypothesis}

\newcommand{\zerarcounters}{\setcounter{equation}{0}\setcounter{theorem}{0}}

\newcommand{\ZZZ}{\mathds{Z}} 
\newcommand{\CCC}{\mathds{C}} 
\newcommand{\NNN}{\mathds{N}} 
\newcommand{\QQQ}{\mathds{Q}} 
\newcommand{\RRR}{\mathds{R}} 
\newcommand{\TTT}{\mathds{T}} 
\newcommand{\uno}{\mathds{1}}

\newcommand{\calF}{{\mathcal F}} 
\newcommand{\calG}{{\mathcal G}} 
\newcommand{\calH}{{\mathcal H}}

\newcommand{\MM}{{\mathcal M}} 
 
\newcommand{\calO}{{\mathcal O}}

\newcommand{\TT}{{\mathcal T}} 
\newcommand{\calU}{{\mathcal U}} 
\newcommand{\calV}{{\mathcal V}}

\newcommand{\gotn}{{\mathfrak n}}

\newcommand{\gotI}{{\mathfrak I}} 
\newcommand{\gotJ}{{\mathfrak J}}

\newcommand{\gotR}{{\mathfrak R}} 
 
\newcommand{\gotT}{{\mathfrak T}}

\newcommand{\matA}{{\mathscr A}}

\newcommand{\matV}{{\mathscr V}}

\newcommand{\io}{\infty} 
\newcommand{\e}{\varepsilon} 
\newcommand{\al}{\alpha}

\newcommand{\x}{\xi}

\newcommand{\om}{\omega}

\newcommand{\oo}{\omega}

\newcommand{\nn}{\nu} 
\newcommand{\pps}{\psi} 
\newcommand{\vzero}{0} 

\newcommand{\der}{{\rm d}} 
\newcommand{\ii}{{\rm i}}

\def\tilde#1{\widetilde{#1}}
 
\def\ins#1#2#3{\vbox to0pt{\kern-#2 \hbox{\kern#1 #3}\vss}\nointerlineskip}

\begin{document}

\title{{\titlefont \textbf{
Response solutions for strongly dissipative\\
quasi-periodically forced systems\\
with arbitrary nonlinearities and frequencies}}}
\author 
{\bf Guido Gentile and Faenia Vaia
\vspace{2mm} 
\\ \small
Dipartimento di Matematica, Universit\`a di Roma Tre, Roma, I-00146, Italy
\\ \small 
E-mail: gentile@mat.uniroma3.it, faenia.vaia@uniroma3.it
}

\date{} 
 
\maketitle 
 
\begin{abstract} 
We consider quasi-periodically systems in the presence of dissipation
and study the existence of response solutions, i.e.~quasi-periodic solutions with
the same frequency vector as the forcing term. When the dissipation is large enough and a
suitable function involving the forcing has a simple zero, response solutions are known to exist
without assuming any non-resonance condition on the frequency vector.
We analyse the case of non-simple zeroes 
and, in order to deal with the small divisors problem, we  confine ourselves to two-dimensional frequency vectors,
so as to use the properties of continued fractions. 
We show that, if the order of the zero is odd (if it is even, in general no response solution exists),
a response solution still exists provided the inverse of the parameter measuring the dissipation
belongs to a set given by the union of infinite intervals depending on the convergents
of the ratio of the two components of the frequency vector.
The intervals may be disjoint and as a consequence
we obtain the existence of response solutions in a set with ``holes''. If we want the set
to be connected we have to require some non-resonance condition on the frequency: in fact,
we need a condition weaker than the Bryuno condition usually considered in small divisors problems.
\end{abstract} 

\zerarcounters 
\section{Introduction} 
\label{sec:1} 

Consider a conservative mechanical system which has a stable equilibrium point corresponding to a strict minimum
of the potential energy.
If a dissipative term proportional to the velocity is introduced, the equilibrium point becomes asymptotically stable.
If a further quasi-periodic forcing term is added to the equations of motion, in general the equilibrium point disappears
and -- by analogy to the the case of the periodically forced systems --  one expects
that a response solution, that is a solution with the same frequencies of the forcing term, arises by bifurcation.

However, contrary to the periodic case, when discussing the existence of a response solution in
quasi-periodically forced systems one has to deal with small divisor problems,
at least if one attempts a perturbation theory approach.
By requiring suitable non-resonance conditions on the frequencies, such as the Diophantine or the Bryuno condition \cite{B},
the small divisors can be controlled.
Still, the question remains what happens if no condition is assumed.
Because of the presence of dissipation, especially if the dissipation is large enough,
one may expect that a response solution exists for all frequencies.

As a simple paradigmatic model for a one-dimensional strongly dissipative forced system,
we consider the singular non-autonomous ordinary differential equation on $\RRR$
\begin{equation} \label{eq:1.1}
\e \ddot x + \dot x + \e g(x) = \e f(\oo t) ,
\end{equation}
where $x\in\matA \subset \RRR$, with $\matA$ an open set, the dots denote derivatives with respect to time $t$,
$\oo \in \RRR^{d}$ is the frequency vector of the forcing term and $\e \in \RRR_{+}$
is the perturbation parameter: $\e$ small means large dissipation.
The functions $g: \matA \to \RRR$ and $f: \TTT^{d} \to \RRR$,
with $\TTT=\RRR/2\pi\ZZZ$, are real analytic, and $f$ is quasi-periodic in $t$, so that it admits the Fourier series expansion
\begin{equation} \label{eq:1.2}
f(\pps)= \sum_{\nn \in \ZZZ^{d}} {\rm e}^{\ii \nn \cdot \pps} f_{\nn} ,
\end{equation}
where $\pps \in \TTT^{d}$; here and henceforth $\cdot$ denotes the standard scalar product in $\RRR^d$,
i.e. $\nn\cdot\pps=\nu_1\psi_1+\ldots+\nu_d\psi_d$.

We are interested in response solutions to \eqref{eq:1.1} which reduce to constants when $\e=0$,
that is which arise by bifurcation from solutions of the unperturbed system;
{\color{black} 
a major hindrance in the problem is that for $\e=0$ the constant solutions
form an open set, while for $\e \neq 0$ the solution is likely to be isolated.
}

Since we do not want to impose any condition on the frequency vector $\oo$, we require only that
its components are non-resonant: in fact, one can always reduce to such a case by possibly
redefining the frequency vector and taking a smaller value of $d$ \cite{GMV}.

If there exists no $x \in\matA$ such that $g(x)=f_{\vzero}$, one checks immediately that no response solution exists
which goes to a constant when $\e$ tends to zero.
Thus, we assume that the function $g(x)-f_{\vzero}$ has a zero $c$. We say that $c$ is a zero of order $\gotn$ if
\begin{equation} \nonumber
g(c)=f_{\vzero} , \qquad \frac{d^\gotn g}{dx^\gotn}(c) \neq 0 , \qquad
\frac{d^j g}{dx^j}(c) = 0 , \quad j=1, \dots, \gotn-1 .  
\end{equation}
If $\gotn$ is even, in general no response solution exists \cite{G1}, so we confine ourselves to odd values of $\gotn$.
If the zero is simple ($\gotn=1$), then we know that a response solution exists for all non-resonant frequencies \cite{GV,WD};
moreover, if $g'(c)>0$, the response solution is a local attractor \cite{G1,GV}.
So we assume the following condition involving the functions $f$ and $g$. 

\begin{hyp}[Non-degeneracy condition] \label{hyp:1}
There exists $c\in \RRR$ such that $x=c$ is a zero of odd order $\gotn > 1$ of the function $g(x)-f_{\vzero}$.
\end{hyp}

We want to study the existence of response solutions to \eqref{eq:1.1} for $\e$ small enough,
by assuming only Hypothesis \ref{hyp:1} on the nonlinearity.
If $f$ is a trigonometric polynomial of degree $N_f$, for any non-resonant $\oo\in\RRR^d$,
a response solution exists for all $\e$ less than a value $\e_0$ which depends on $N_f$ as well on $\al$ \cite{CFG2,V};
of course $\e_0$ goes to $0$ when $N_f$ goes to infinity. Therefore we are interested in extending the result to arbitrary
analytic functions $f$ or, equivalently, to find bounds which are uniform in $N_f$.
To deal with the small divisors we shall need to rely on the theory of continued fraction.
Thus we will restrict the analysis to the two-dimensional case and make the following assumption.

\begin{hyp}[\textbf{Non-resonance condition}] \label{hyp:2}
The frequency vector $\oo$ is in $\RRR^2$ and has rationally independent components, that is
$\oo \cdot \nn \ne 0$ $\forall \nn \in \ZZZ^{2}_{*}:=\ZZZ^{2}\setminus\{\vzero\}$.
\end{hyp}

Without loss of the generality we may and do assume that the frequency vector is of the form $\oo=(1,\al)$,
with $\al\in\RRR\setminus\QQQ$.

We shall see that in general, if we do not make any further assumption on $\oo$,
we do not obtain the existence of response solutions for all values of $\e$ small enough, as 
we might have to exclude a set of possibly infinitely many  intervals accumulating to the origin.
In other words, we find that the response solution exists -- or, at least, can be proved to exist --
only for $\e$ in a set with holes, whose number and sizes depend on the frequency vector.
Nevertheless, if we assume a suitable non-resonance condition, weaker than Bryuno's condition,
we are able to make the holes to disappear and recover the existence result for all $\e$ small enough.
More formal statements will be provided in the next section.

\zerarcounters 
\section{Main result}
\label{sec:2}

If $\Sigma_{\xi}:=\{ \psi \in \CCC^d : \hbox{Re}(\psi_i) \in \TTT , \, |\hbox{Im}(\psi_i)| \le \xi , \, i=\,\ldots,d\}$
denotes the strip where $f$ is analytic, the Fourier coefficients $f_{\nn}$ in \eqref{eq:1.2} satisfy the bound
\begin{equation} \label{eq:2.1}
|f_{\nn}| \le \Phi \, {\rm e}^{-\xi |\nn|} \qquad \forall \nn\in\ZZZ^d ,
\end{equation}
for a suitable positive constant $\Phi$.

Let us denote by $\Delta(z, \rho)$ the disk of center $z$ and radius $\rho$ in the complex plane.
Because of the assumption of analyticity on $g$, for any $c\in \matA$
there exists $\rho_{0}>0$ such that $g$ is analytic in $\Delta(c, \rho_{0})$.
Then for all $\rho < \rho_{0}$, if one defines $\Gamma:= \max\{|g(x)|: x \in \Delta(c, \rho)\}$, one has, under Hypothesis \ref{hyp:1},
\begin{equation} \label{eq:2.2}
g(x)= g(c) + \sum_{p=\gotn}^{\infty} g_{p} (x - c)^p, \qquad g_{p}:=\frac{1}{p!}
\frac{d^p g}{dx^p}(c) , \qquad |g_{p}| \le \Gamma \rho^{-p},
\end{equation}
where we have used Cauchy's estimates to bound the derivatives.

Fixed $\al\in\RRR\setminus\QQQ$, let  $p_{n}/q_{n}$ be the convergents of $\alpha$ (see Section \ref{sec:4} for details). 
Given two positive constants $C$ and $C'$, for all $n\in\NNN$ such that ${\rm e}^{-C' q_{n'}} \le (C q_{n'})^{-\gotn-1}$ for all $n'\ge n$ define
\begin{equation} \label{eq:2.3}
I_{n}(C,C') := \Big[ {\rm e}^{-C' q_{n}} , (C q_{n})^{-\gotn-1}\Big], \qquad
J_{n}(C,C') := 
{\color{black}
\mbox{Cl}
}
\left( \bigcup_{n' \ge n} I_{n'}(C,C') \right) ,
\end{equation}
{\color{black}
where $\hbox{Cl}$ denotes set closure.}
Note that for any values of $C$ and $C'$, the intervals $I_{n}(C,C')$ in \eqref{eq:2.3} are well defined for $n$ large enough,
{\color{black}
and that $0 \in J_{n}(C,C')$.}
In the following we fix $C'=C_{0}$, with
\begin{equation} \label{eq:2.4}
C_{0} =  \frac{\gotn+1}{4(\gotn^{2}+2\gotn-1)}\xi .
\end{equation}

We can rephrase the problem we have addressed in Section \ref{sec:1} as follows: 
under Hypotheses \ref{hyp:1} and \ref{hyp:2}, we look for solutions to \eqref{eq:1.1} of the form
\begin{equation} \label{eq:2.5}
x(t,\e)= c + X(\om t , \e) , 
\end{equation}
where 
$X(\om t, \e)$ is a quasi-periodic function which goes to 0 as $\e$ tends to 0.
We shall prove the following result.

\begin{thm} \label{thm:1}
Consider the ordinary differential equation \eqref{eq:1.1}, with $f$ analytic in the strip $\Sigma_{\xi}$ and $\oo=(1, \alpha)$,
and assume Hypotheses \ref{hyp:1} and \ref{hyp:2}. Denote by $p_{n}/q_{n}$ the convergents of $\alpha$
and let $C_{0}$ be fixed as in \eqref{eq:2.4}.
Then for any constant $C_1$ there exists 
$N\in \NNN$ such that
for all $\e \in J_{N}(C_1,C_{0})$ there is at least one quasi-periodic solution $x(t,\e)= c + X(\oo t, \e)$ to \eqref{eq:1.1}, such that
the function $X(\pps, \e)$ is analytic in $\pps$ in the strip $\Sigma_{\xi'}$, with $\xi'<\xi/4$, is continuous in $\e$ in the sense of Whitney
and goes to zero as $\e \to 0$. 
\end{thm} 

{\color{black}
We refer to refs.~\cite{dVP, Whitney} for the definition of continuity -- and differentiability -- in the sense of Whitney; here we just recall that
a function which is defined and continuous in a closed subset of $\RRR$
can be extended to a function continuous everywhere in $\RRR$.
(See also refs.~\cite{BS,FS} for a similar use of the notion of continuity in the sense of Whitney in different contexts).
}

As we shall see along the proof of the theorem, for fixed $C_1$, we write the response solution as a series depending on $\e$
and prove that there exists $\e_0>0$ small enough such that the series converges provided $|\e|<\e_0$ and $\e\in I_n(C_1,C_{0})$
for some $n\in\NNN$. The constant $\e_0$ depends on $C_1$, more precisely it is of the form $\e_0=C_1^{\gotn(\gotn+1)}\eta_0$,
with $\eta_0$ depending on all the other parameters but $C_1$ -- i.e.~$\xi$, $\rho$, $\Phi$ and $\Gamma$.
Once $\e_0$ has been determined, the constant $N$ must be taken so that
 \begin{subequations} \label{eq:2.6}
\begin{align} 
\frac{1}{(C_{1} q_{N})^{\gotn+1}} &\le \e_{0},
\label{eq:2.6a} \\
\frac{1}{(C_{1}q_{n})^{\gotn+1}} & \ge {\rm e}^{- C_{0}q_{n}} \quad \forall n \ge N ,
\label{eq:2.6b}
\end{align}
\end{subequations}
which ensure the convergence of the series for all $\e$ in the set $J_N(C_1,C_{0})$.
Note that \eqref{eq:2.6b} is satisfied for all $n$ if $C_1$ is taken small enough and
for all positive $C_1$ if $n$ is large enough.

If one looks for optimal bounds, one has to choose $C_1$ in order to make $\e_0$ as large as possible and,
at the same time, reduce the sizes of the holes (see Section \ref{sec:10}).
Indeed, the intervals $I_n(C_1,C_{0})$ may be disjoint: since there is in general no a priori relation between
${\rm e}^{-C_{0} q_{n}} $ and $(C_{1} q_{n+1})^{-\gotn-1}$, it may happen that for some $n \ge N$
the intervals $I_{n}(C_1,C_{0})$ and $I_{n+1}(C_1,C_{0})$ are as represented in Figure \ref{rettaaa}.

\begin{figure}[h]
\vspace{-1.6cm}
	\centering
	\ins{168pt}{-78pt}{$\frac{1}{(C_{1}q_{n+1})^{\gotn+1}}$}
	\ins{265pt}{-78pt}{$\frac{1}{(C_{1}q_{n})^{\gotn+1}}$}
	\ins{220pt}{-78pt}{${\rm e}^{-C_{0}q_{n}} $}
	\ins{126pt}{-78pt}{${\rm e}^{-C_{0}q_{n+1}} $}
	\ins{100pt}{-63pt}{$\dots$}
	\ins{280pt}{-58pt}{$\e_{0}$}
	\ins{252pt}{-53pt}{$I_{n}$}
	\ins{160pt}{-53pt}{$I_{n+1}$}
         \null\hspace{-2.4cm}
	\subfigure{\includegraphics[scale=0.45]%
		{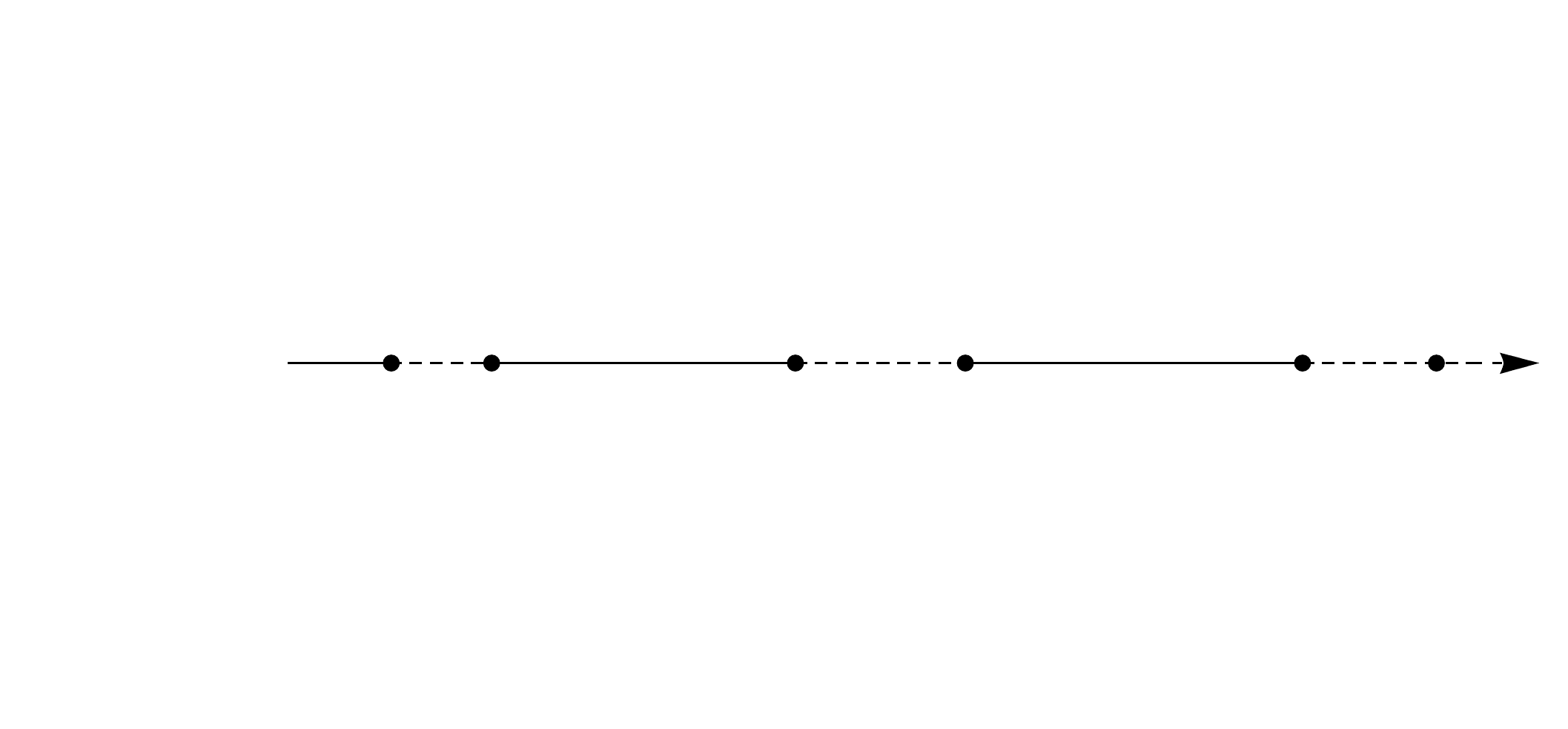}}
\vspace{-1.4cm}
	\caption{Interval $I_n=I_{n}(C_1,C_{0})$ and $I_{n+1}=I_{n+1}(C_1,C_{0})$ in the case they are disjoint.}
	\label{rettaaa}
\end{figure}

The overall measure of the allowed intervals depends on the irrational number $\alpha$, in particular on its convergents $p_{n}/q_{n}$,
so we cannot  say in general whether it is either large or small.  Once $N >0$ has been fixed as in Theorem \ref{thm:1},
a natural problem to address is whether it is possible to obtain the result of existence for all $|\e| < \e_{0}$ by imposing
some further condition on $\oo$. This is equivalent to requiring
\begin{equation} \label{eq:2.7}
{\rm e}^{-C_{0} q_{n}} \le (C_{1} q_{n+1})^{-\gotn-1}, \qquad n \ge N,
\end{equation}
so as to have the situation represented in Figure \ref{senzabuchi} for all $n\ge N$.

\begin{figure}[h]
\vspace{-0.4cm}
	\centering
	\ins{295pt}{-43pt}{$\e_{0}$}
	\ins{267pt}{-58pt}{$\frac{1}{(C_{1}q_{n})^{\gotn+1}}$}
	\ins{203pt}{-70pt}{$\frac{1}{(C_{1}q_{n+1})^{\gotn+1}} $}
	\ins{185pt}{-58pt}{${\rm e}^{-C_{0}q_{n}}$}
	\ins{118pt}{-70pt}{$\frac{1}{(C_{1}q_{n+2})^{\gotn+1}} $}
	\ins{105pt}{-58pt}{${\rm e}^{-C_{0}q_{n+1}}$}
	\ins{107pt}{-20pt}{$I_{n+2}$}
	\ins{233pt}{-30pt}{$I_{n}$}
	\ins{158pt}{-24pt}{$I_{n+1}$}
	\null\hspace{1cm}
	\subfigure{\includegraphics[scale=0.45]%
		{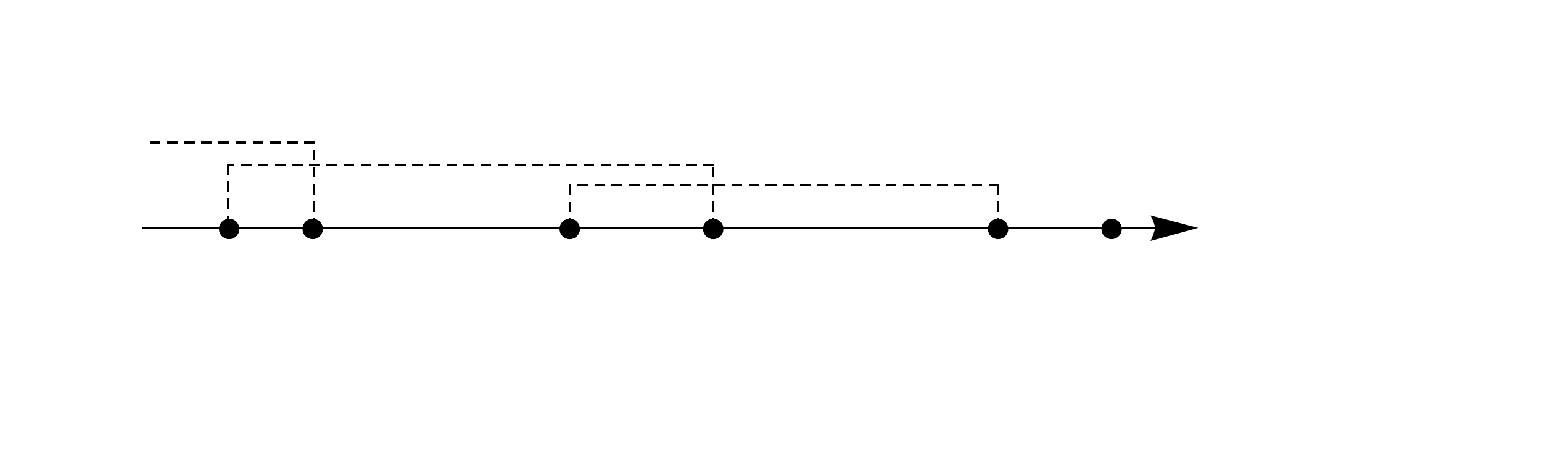}}
\vspace{-0.4cm}
	\caption{Intervals $I_n=I_{n}(C_1,C_{0})$, $n\ge N$, in the case in which there are no holes.}
	\label{senzabuchi}
\end{figure}

From the results available in the literature \cite{GBD1,GBD2,G2}, we know that, taken $\e_{0}>0$,  if $\oo$ is Diophantine or Bryuno,
response solutions exist for all $\e$ in a set without holes \cite{G1,CFG2}.
Hence it is not surprising that \eqref{eq:2.7}  is satisfied when $\oo$ is Diophantine or Bryuno. 

Indeed, if $\oo$ is Diophantine,  i.e.~if there exist two constants $\gamma>0$ and $\tau\ge 1$ such that
$|\oo\cdot\nn| > \gamma|\nu|^{-\tau}$ for all $\nn\in \ZZZ^2_*$, condition \eqref{eq:2.7} easily follows by using that
\begin{equation} \label{eq:2.8} 
\frac{1}{(C_{1} q_{n+1})^{\gotn+1}} \ge \frac{1}{(K_{0}C_{1}q_{n}^{\tau})^{\gotn+1}} >  {\rm e}^{-C_{0} q_{n}} , \qquad n \ge N , 
\end{equation}
with $K_{0}:=\gamma^{-1}(1+4\al^2)^{\tau/2}$, $C_{0}$ as in \eqref{eq:2.1} and $C_{1}$ and $N$ suitably chosen.

If $\oo$ is a Bryuno vector, i.e.~if the sequence
\begin{equation} \label{eq:2.9} 
\e_{n}(\alpha):= \frac{1}{q_{n}} \log q_{n+1}
\end{equation}
is summable \cite{B,MMY,M}, then \eqref{eq:2.7} is satisfied once more, because
\begin{equation} \label{eq:2.10} 
\begin{split}
\log(C_{1}q_{n+1})^{\gotn+1} &=(\gotn+1) \log C_{1} + (\gotn+1) \log q_{\gotn+1} \\
&= (\gotn+1) \log C_{1} + (\gotn+1) q_{n}\e_{n}(\alpha) < C_{0} q_{n}
\end{split}
\end{equation}
with $C_{0}$ as in \eqref{eq:2.3} and $C_{1}$ chosen according to the Theorem.
A better characterisation of the vectors for which there are no holes is provided in Section \ref{sec:10}.

In general, the opposite is not true: there are vectors which satisfy \eqref{eq:2.7}, but are not Bryuno vectors.
In fact, it is sufficient that $\e_n(\alpha)$ be small enough so as to satisfy the bound \eqref{eq:2.10},
for some constant $C_1$ possibly depending on $n$. Again we refer to Section \ref{sec:10} for further details.

To summarise, for $\gotn \ge 3$ (odd) and $d=2$, if we do not impose \eqref{eq:2.7} on the convergents of $\alpha$,
we find that the response solutions exist in a set with holes; instead if we require \eqref{eq:2.7}
we have the existence of response solutions for all $|\e|<\e_{0}$ and for a class of frequency vectors
which satisfy a condition weaker than the Bryuno condition -- and also weaker than the condition that $\e_{n}(\alpha) \to 0$ as $n\to\io$.
The latter result improves, for $d=2$, the condition on $\oo$ assumed in previous results available in the literature,
where stronger non-resonance conditions were considered for the frequency vector \cite{CCD1,G1,G2}.

It would be interesting to investigate whether the Theorem above can be extended to higher dimensional frequency vectors
and if existence of a response solution can be proved for values of $\e$ inside the holes left out by the Theorem,
as it happens in the case $\gotn=1$, where the response solutions are known to exist for any $d$ and for $\e$ in set with no holes \cite{GV,WD}.
Another generalisation to look for, in the case $\gotn >1$, is to higher dimensions for the space variable $x$,
as already studied in the case of simple zeroes \cite{GMV,WD}.

\zerarcounters 
\section{Some related results in the literature}
\label{sec:3}

Existence of quasi-periodic and almost-periodic solutions to ordinary differential equations 
in problems where no hypothesis is made on the frequencies has not been studied in the literature
as extensively as in the case in which one requires some non-resonance condition such as the Diophantine or the Bryuno condition.
In the framework of KAM theory, both rigorous results \cite{Fo,Be} and strong numerical evidence \cite{CCV,LL,O,T} suggest
that, in general, the invariant tori of the unperturbed system which are close to resonance break up when the perturbation is switched on;
therefore, if one is looking for results holding for all frequencies
one has to consider either conservative systems away from the KAM regime \cite{BC1,BC2,BZ} or non-conservative systems.

Typically response solutions in forced systems arise by bifurcation.
Bifurcation phenomena have been widely investigated in the literature -- see for instance the works by
Broer, Han{\ss}mann and coauthors \cite{BV,BGV,BHJVW,CLB,H1,H2,H3,H4}.
If no condition is assumed on the forcing frequencies, in order to bypass the small divisor problem,
a non-degeneracy condition is generally assumed on the bifurcating solution,
usually a condition of hyperbolicity or exponential dichotomy \cite{BB,Co,F,He,XLC};
for example in ref.~\cite{BB} an invariant torus bifurcates from the equilibrium point under a suitable assumption of hyperbolicity.

On the contrary, in the problem addressed in the present paper,
the unperturbed bifurcating solution is not given a priori and \emph{a fortiori} no stability or hyperbolicity condition is assumed:
any constant $c$ is a solution to \eqref{eq:1.1} when $\e=0$ and the existence of response solution is proved
for a special value of the constant. In fact, equation \eqref{eq:1.1} is a particular case of the equation
\begin{equation}
\label{eq:3.1}
\dot{x}= f(t,x,y,\e), \qquad \e \dot{y} =g(t,x,y,\e),
\end{equation}
with $\e>0$, $(x,y)\in \RRR^{n}\times \RRR^{m}$ and $f,g$ almost periodic in time,
which is studied in refs.~\cite{C,HS}. However,  the point of view considered in refs.~\cite{C, HS} is different
since the assumptions on the almost periodic solution of the unperturbed system
\begin{equation}
\label{eq:3.2}
\dot{x}= f(t,x,y,0), \qquad 0 =g(t,x,y,0),
\end{equation}
in particular the hypothesis that a non-degenerate almost periodic solution exists,
make the problem is of hyperbolic type and no small divisors appear.
Analogous results for  equations of the form \eqref{eq:3.1} are also provided in similar models
in both the periodic \cite{A,F,FL,Ve,W} and almost periodic context \cite{Sm}.

As we said in the previous sections, we will use the theory of continued fractions
to bound the small divisors in the case of analytic $f$:
the continued fractions are used to deal with Liouvillean frequncies in refs.~\cite{LG, WYZ}. 
In ref.~\cite{LG} the following equation is studied:
\begin{equation}
\label{eq:3.3}
\ddot{x} + \lambda^{2} x = \e F(\oo t, x, \dot{x}), \quad x \in \RRR, \quad \oo=(1, \alpha), \,\, \alpha \in \RRR\setminus\QQQ,
\end{equation}
where $\e,\, \lambda$ are positive parameters and $F: \TTT^{2} \times \RRR\times \RRR \to \RRR $
is a real analytic function satisfying the reversibility condition
\begin{equation}
\label{eq:3.4}
F(\pps,x,y)= F(-\pps, x, -y).
\end{equation}
Without assuming $\oo$ to satisfy any Diophantine or Bryuno conditions,
for any closed interval $\calO \subset \RRR \setminus \{0\}$ and any sufficiently small $\gamma >0$, 
a response solution to \eqref{eq:3.4} is proved to exist for $\e$ small enough and for $\lambda$
in a Cantor set $\mathcal{O}_{\gamma} \subset \calO$ with large relative Lebesgue measure.
The equation considered in \cite{WYZ} is  different with respect to \eqref{eq:3.3} since the function $F$ does not depend on $\dot{x}$.
Both \cite{LG, WYZ} use the properties of the continued fractions theory and  exploit the so-called
CD-bridge between the denominators of the best rational approximations of continued fractions \cite{AFK}.
Once more our problem is of a different kind, since \eqref{eq:1.1}
is a singular differential equation and does not verify the reversibility condition in \eqref{eq:3.4}

\zerarcounters
\section{Strategy of the proof}
\label{sec:4}

Let us write in \eqref{eq:2.5}
\begin{equation} \label{eq:4.1}
X(\om t,\e) = \zeta + u(\oo t; \e,\zeta,c) , \qquad u(\pps; \e,\zeta,c) = 
\sum_{\nn\in\ZZZ^{d}_{*}} {\rm e}^{\ii\nn\cdot\pps} u_{\nn} ,
\end{equation}
where $\ZZZ^d_*:=\ZZZ^d\setminus\{\vzero\}$, $\zeta$ is a real parameter to be fixed
and $\pps \mapsto u(\pps;\varepsilon,\zeta,c)$ is a zero-average quasi-periodic function,
with Fourier coefficients depending on $\varepsilon$, $\zeta$ and $c$ -- even though
we shall not make explicit such a dependence in order not to overwhelm the notations.

We can write \eqref{eq:1.1} in Fourier space as
\begin{equation} \label{eq:4.2}
(\ii \oo \cdot \nn)(1 + \ii \e \oo \cdot \nn)u_{\nn} + \e [g(c + \zeta + u)]_{\nn} = \e f_{\nn} \qquad \nn \ne \vzero,
\end{equation}
provided for $\nn=\vzero$ one has
\begin{equation} \label{eq:4.3}
[g(c + \zeta + u)]_{\vzero} = f_{\vzero}.
\end{equation}
%
The notation $[F(c+\zeta+u)]_{\nn}$, for any function $F(x)$ analytic in its argument, means that
we first consider the Taylor expansion of $F(c+\zeta+u)$ about the point $c$,
then write $u$ as a Fourier series according to \eqref{eq:4.1}, and finally keep the Fourier coefficient with index $\nn$:
\begin{equation} \label{eq:4.4}
[F(c + \zeta + u)]_{\nn} = \sum_{p=0}^{\io} \frac{1}{p!} \frac{\der^p F}{\der x^p} (c)
\sum_{q=0}^{\io} \left( \begin{matrix} p \\ q \end{matrix} \right) \zeta^{p-q} \!\!\!\!
\sum_{\substack{\nn_1,\ldots,\nn_q\in\ZZZ_*^d \\ \nn_1+\ldots+\nn_q =\nn}} \!\!\!\! u_{\nn_1} \ldots u_{\nn_q} .
\end{equation}
We call \eqref{eq:4.2} the \emph{range equation} and \eqref{eq:4.3} the \emph{bifurcation equation},
see also \cite{CH}. We first study the range equation looking for a solution to \eqref{eq:4.2},
depending on the parameter $\zeta$ that is supposed to be close enough to zero.
Then we analyse \eqref{eq:4.3} and fix $\zeta$ in order to make such an equation to be satisfied as well. 

{\color{black}
We find that the parameter $\zeta$, in general, is no more than continuous in $\e$.
In our case, this suffices, since all we need is to prove that the parameter goes to $0$ as $\e$ tends to to zero.
Of course, in principle more regularity is possible, and the parameter could have different branches,
as found in similar contexts when bifurcation phenomena occur. One could even think that a fractional power series
may be constructed: for instance, this happens for both the Melnikov problem \cite{CG}
and lower dimensional tori of codimension 1 \cite{GGG}, when the case of higher order zeroes is considered.
For the problem under study the situation is more delicate, since already in the case of simple zeroes
in general no more than continuity is found \cite{GV}. In any case, we do not exclude that stronger regularity results
may be obtained, possibly with different methods; see also ref.~\cite{CCD1,CFG1,CCD2} for some results
about the form of the analyticity domains in dissipative perturbations of Hamiltonian systems.
}

\zerarcounters
\section{A brief review on the continued fractions}
\label{sec:5}

In this section we review some basic properties of the continued fraction that we shall use later on;
we refer to refs.~\cite{HW,K,S} for details and proofs. A \emph{finite continued fraction} is an expression of the form
\begin{equation} \nonumber
[a_{0}, \dots, a_{n}] = a_{0} + \cfrac{1}{a_{1}+ \cfrac{1}{a_{2} + \cfrac{1}{\ddots +\cfrac{1}{a_{n}}}}}
\end{equation}
with $a_{0} \in \RRR$ and $a_{1}, \dots,a_{n} \in \RRR_{+}$;
the corresponding coefficients $a_{0},\dots, a_{n}$  are called the \emph{partial quotients} of the finite continued fraction.
An \emph{infinite continued fraction} is defined as the limit for $n\to \infty$ of $x_{n} :=  [a_{0},\dots,a_{n}]$, i.e.
\begin{equation} \nonumber
x = [a_{0}, a_{1}, a_{2}, \dots] = 
\lim_{n\to\infty} [a_{0},\dots,a_{n}] = a_{0} + \cfrac{1}{a_{1}+ \cfrac{1}{a_{2} + \cfrac{1}{\ddots +\cfrac{1}{\ddots}}}}
\end{equation}
when the limit exists; in such a case we say that the continued fraction converges.
A continued fraction is called \emph{simple} if
{\color{black}
$a_0\in\ZZZ$ and
$a_i\in\NNN$ for all $i \ge 1$.}
Any irrational number $x \in \RRR\setminus \QQQ$ is represented as a unique simple infinite continued fraction.

%
	%
	%
	%
	%
	%
	%
	%
	%
	%
	%

\begin{prop}
	\label{prop:5.1}
	Given a simple continued fraction $[a_{0}, a_{1}, a_{2}, \dots]$, 
	define
	\vspace{-.2cm}
	\begin{eqnarray}
	& & p_{0} =a_{0}, \qquad p_{1}= a_{1}a_{0} +1, \qquad p_{k}= a_{k} p_{k-1} + p_{k-2}, \qquad k\ge 2 , \nonumber \\
	%
	& & q_{0} =1, \qquad q_{1}= a_{1}, \qquad q_{k}= a_{k} q_{k-1} + q_{k-2}, \qquad  k \ge 2. \nonumber
	\end{eqnarray}
	Then one has
	\begin{equation} \nonumber
	x_{k}:= \frac{p_{k}}{q_{k}}= [a_{0}, \dots, a_{k}], \qquad  k \ge 0. 
	\end{equation}
	We call $x_k$ the $k$-th \emph{convergent} of the continued fraction.
\end{prop}

\begin{prop}[Properties of the convergents]
	\label{prop:5.2}
	Given $x\in\RRR\setminus\QQQ$, let $p_{k}, q_{k}, x_{k}$ be defined as in Proposition \ref{prop:5.1}. One has:
	\vspace{-.2cm}
	\begin{enumerate}
	\itemsep0em
		\item[\emph{(1)}] $q_{k+1} > q_{k} > 0$ $\forall k > 0$;
		\item[\emph{(2)}] $p_{k} > 0\, [p_{k}<0]$ $\forall k >0$ if $\,x>0 \, [x<0]$;
		\item[\emph{(3)}] $p_{k}q_{k-1} - p_{k-1}q_{k} = (-1)^{k-1}\,$ for $\, k\ge 1$
		and $p_{k}q_{k-2} - p_{k-2}q_{k} = (-1)^{k}a_{k}\,$ for $\, k \ge 2$;
		\item[\emph{(4)}] $x_{2k}$ is strictly increasing;
		\item[\emph{(5)}] $x_{2k+1}$ is strictly decreasing;
		\item[\emph{(6)}]  $x_{2k} < x < x_{2k+1}$ for all $k\in \NNN$.
	\end{enumerate}
\end{prop}

\begin{prop}
	\label{prop:5.3}
	Let $x $ be an irrational number and let $\{p_{k}/q_{k}\}$ be the convergents of the continued
	{\color{black}
	simple}
	fraction representing $x$. Then 
	for all $k\ge 0$ one has
	\begin{equation} \nonumber
	\frac{1}{q_{k}(q_{k} + q_{k+1})} < \Big|x - \frac{p_{k}}{q_{k}}\Big| < \frac{1}{q_{k}q_{k+1}} .
	\end{equation}
\end{prop}

	A rational number $p/q$, with $q>0$ and  \emph{GCD}$(p,q)=1$, is called a \emph{best rational approximation} of $x \in \RRR \setminus \QQQ$ if 
	\begin{equation} \nonumber
	|nx - m| > |qx -p|, \qquad \forall m,n \in \ZZZ : 0 < |n| \le q \,\, \text{and}\,\, \frac{m}{n}\ne \frac{p}{q}.
	\end{equation}
	%
%

\begin{prop}
	\label{prop:5.4}
	Any rational number $p/q$ is one of the best rational approximation if and only if $p/q$ is a convergent of $x$.
\end{prop}

\begin{prop}
	\label{prop:5.5}
	The convergents $\{p_{n}/q_{n}\}$ of $x$ satisfy the inequalities
	\begin{equation} \nonumber
	|q_{0}x-p_{0}| > |q_{1} x - p_{1}|> \dots > |q_{k} x -p_{k}|> \dots
	\end{equation}
	\end{prop}


We want to use the theory of continued fractions in order to deal with the small divisors problem. 
To this end, we shall restrict ourselves to two-dimensional frequency vectors $\oo$;
besides, without any loss of generality, we assume $\oo$ to be the form $\oo:=(1,\alpha) \in \RRR^{2}$,
with $\alpha \in \RRR\setminus \QQQ$ (to comply with Hypothesis \ref{hyp:2}).

\begin{prop} \label{prop:5.6}
Let $\{p_{n}/q_{n}\}$ be the convergents of the
{\color{black}
simple}
continued fraction representing $\al$.
For all $\nn:=(\nu_{1}, \nu_{2})\in \ZZZ^{2}\setminus \{\vzero\}$
such that $0<|\nu_{2}|< q_{n}$ and $|\nu_{2}|\ne q_{n-1}$, one has 
\begin{equation} \nonumber
|\oo\cdot\nn|= |\nu_1 + \nu_{2} \alpha| > |\alpha q_{n-1} - p_{n-1}| > \frac{1}{2q_{n}} \quad \forall n \ge 1 .
\end{equation}
\end{prop}

\begin{proof}
The first inequality follows from proposition \ref{prop:5.4}, while the second one follows from proposition \ref{prop:5.3} and property 1
in proposition \ref{prop:5.2}.
\end{proof}

\zerarcounters 
\section{Tree formalism}
\label{sec:5}

A \textit{rooted tree} $\vartheta$ is an acyclic planar graph such that all the lines are oriented toward a unique
point, that we call the \emph{root}, which has only one incident line, the \emph{root line}.
All the points in $\vartheta$, except the root, are called \textit{nodes}.
The orientation of the lines in $\vartheta$ induces a partial ordering 
relation ($\preceq$) between the nodes. Given two nodes $v$ and $w$,
we shall write $w \prec v$ every time $v$ is along the path
(of lines) which connects $w$ to the root; we shall write $w\prec \ell$ if
$w\preceq v$, where  $v$ is the unique node that the line $\ell$ enters -- see Figure \ref{fig1}. 
For any node $v$, denote by $p_{v}$ the number of lines entering $v$.
Given a rooted tree $\vartheta$, we call \emph{first node} the node the root line exits and denote by $N(\vartheta)$ the set of nodes,
by $E(\vartheta)$ the set of \textit{end nodes}, i.e. nodes $v$ with $p_{v}=0$,
by $V(\vartheta)=N(\vartheta) \setminus E(\vartheta)$ the set of \textit{internal nodes}
and by $L(\vartheta)$ the set of lines. We impose the constraint $p_{v}\ge \gotn,$ $\forall v\in V(\vartheta)$.
We define the \textit{order} of $\vartheta$ as $k \equiv k(\vartheta):=|N(\vartheta)|$,
where $|A|$ denotes the cardinality of the set $A$. We refer to refs.~\cite{Bo,HP} for an introduction to graph theory
and to ref.~\cite{G3} for an overview on the tree formalism. 

\begin{figure}[H]
	\centering
	\ins{155pt}{-091pt}{$v_{0}$}
	\ins{219pt}{-060pt}{$v_{1}$}
	\ins{230pt}{-090pt}{$v_{2}$}
	\ins{219pt}{-123pt}{$v_{3}$}
	\ins{270pt}{-009pt}{$v_{4}$}
	\ins{286pt}{-029pt}{$v_{5}$}
	\ins{300pt}{-60pt}{$v_{6}$}
	\ins{285pt}{-091pt}{$v_{7}$}
	\ins{266pt}{-110pt}{$v_{8}$}
	\ins{350pt}{-060pt}{$v_{9}$}
	\ins{360pt}{-091pt}{$v_{10}$}
	\ins{345pt}{-123pt}{$v_{11}$}
	\subfigure{\includegraphics[scale=0.7]%
		{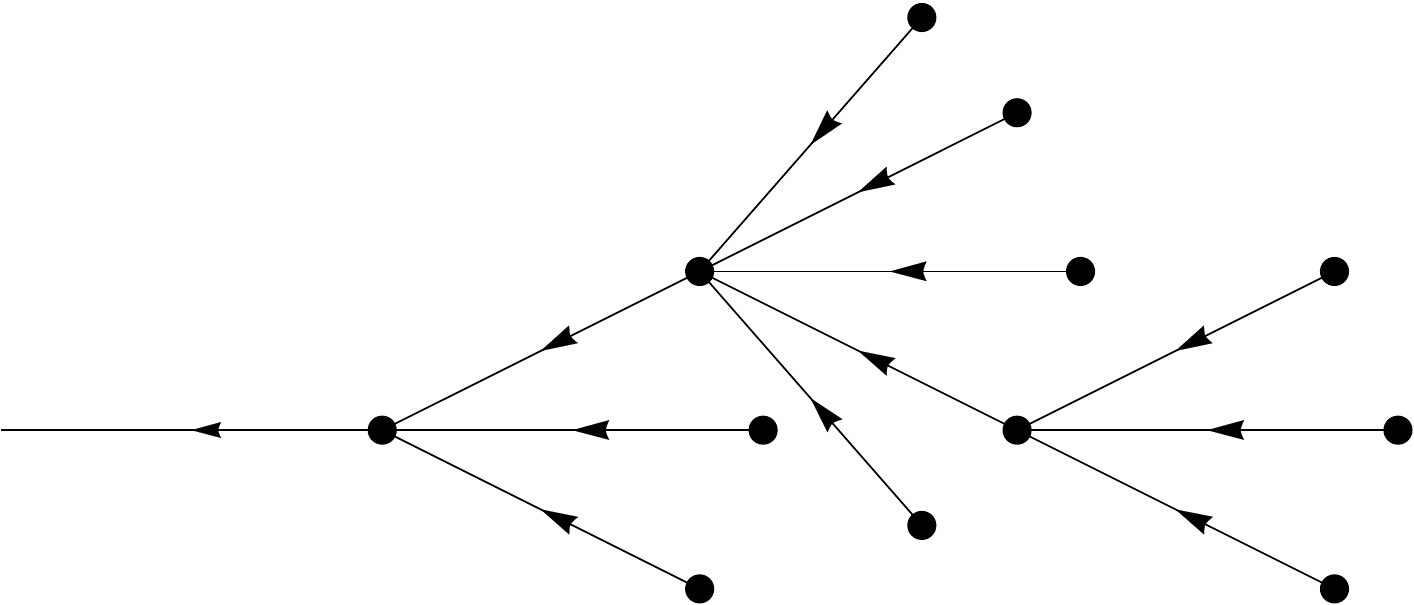}}
	\caption{A tree with 10 vertices, 9 end nodes and 3 internal nodes, for $\gotn =3$.}
	\label{fig1}
\end{figure}

For any node $w\in N(\vartheta)$ the line $\ell$ exiting $w$ can be considered as the root line
of a tree $\vartheta_w$ formed by the nodes $v \in N(\vartheta)$ such that $v\preceq w$, by the lines which connect such nodes and
by $\ell$ itself; such a tree is called the \emph{subtree} of $\vartheta$ with first node $w$.

We associate with each end node $v\in E(\vartheta)$ a \textit{mode} label $\nn_{v}\in\ZZZ^{d}$ and
we split $E(\vartheta)$ in two complementary sets:
$E_{0}(\vartheta)=\{ v \in E(\vartheta) : \nn_{v} = \vzero \}$ and
$E_{1}(\vartheta)=\{ v \in E(\vartheta) : \nn_{v} \neq \vzero \}$, in such a way that $E(\vartheta)=E_{0}(\vartheta) \sqcup E_{1}(\vartheta)$.
With each line $\ell\in L(\vartheta)$ we associate
a \textit{momentum} $\nn_{\ell} \in \ZZZ^{d}$ with the constraint (\emph{conservation law})
\begin{equation} 
\label{eq:6.1}
\nn_{\ell}=\sum_{\substack{w\in E_{1}(\vartheta) \\ w \prec \ell}} \nn_{w},
\end{equation}
i.e.~the momentum of the line $\ell$ is the sum of the mode labels associated with the end nodes preceding $\ell$. 
Equivalently, the momentum of the line $\ell$ which exits a node $v$ is the mode label of $v$ if $v$ is an end node,
and is the sum of the momenta of the lines entering $v$ if $v$ is an internal node.
We split the set $L(\vartheta)$ in two disjoint sets, $L_{0}(\vartheta):=\{\ell \in L(\vartheta): \nn_{\ell}=0\}$
and $ L_{1}(\vartheta):=\{\ell \in L(\vartheta): \nn_{\ell} \ne 0\} = L(\vartheta)\setminus L_{0}(\vartheta)$.

With each line $\ell \in L(\vartheta)$ we also assign a \emph{scale label} $\tilde{n}_{\ell} \in \{0,1\}$.
More generally one can consider the scale label as a number in $\ZZZ$ (we refer to ref.~\cite{G3} for a general overview on multiscale analysis),
but, for the problem we are studying, we just need the scale label to be either zero or one.   

\begin{defi}[Labelled trees]
\label{def:6.1}
A \emph{labelled rooted tree} is a rooted tree with the labels associated with $N(\vartheta)$ and $L(\vartheta)$. 
\end{defi}

\begin{defi}[Equivalent trees]
\label{def:6.2}
We call \emph{equivalent}  two labelled rooted trees which can be transformed into each othe
by continuously deforming the lines in such a way that they do not cross each other. 
\end{defi}

In fact, we shall consider only nonequivalent labelled rooted trees. We denote by
$\TT_{k,\nn}$  the set of nonequivalent trees of order $k$ and momentum $\nn$ associated with the root line.

\begin{defi}[Cluster]
\label{def:6.3}
A \emph{cluster} $T$ on scale $\tilde{n}$ 
is a maximal set of nodes and lines connecting them such that all the lines have scales $\tilde{n}'\le \tilde{n}$
and there is at least one line with scale $\tilde{n}$. The lines entering the cluster $T$ and the possible unique line coming out from it
(if existing at all) are called the \emph{external lines} of the cluster $T$.
\end{defi}

We denote by $V(T)$, $E(T)$ and $L(T)$ the set of internal nodes, of end nodes and of lines of the cluster $T$ respectively,
with the convention that the external lines of $T$ do not belong to $L(T)$. 
In particular we denote with $E_{1}(T)$ and $E_{0}(T)$ the following sets: 
$E_{1}(T):= \{v \in E(T) : \nn_{v}\ne \vzero \}$ and $E_{0}(T):= E(T) \setminus E_{1}(T)$.

In fact, since thare are only two possible scales,
we will only deal with clusters on scale $0$, because the only possible cluster on scale $1$
is the whole tree $\TT_{k, \nn}$ where there is at least one line on scale $1$. 

\begin{defi}[Self-energy cluster]
\label{def:6.4}
A \emph{self-energy cluster} is a cluster $T$ (on scale $0$) such that $T$ has only one entering line which
has the same momentum of exiting line. 
We denote by $\gotR_{0}$ the set of  self-energy clusters.
\end{defi}

According to the definition \ref{def:6.4}, the mode labels associated with the end nodes in a self-energy cluster $T$ are such that 
\[
\sum_{v \in E(T)} \nn_{v} = \vzero.
\]
%

\begin{defi}[Renormalised tree]
\label{def:6.5}
A \emph{renormalised tree} is a tree which does not contain any self-energy clusters.
We denote by $\gotT_{k,\nn}$ the set of renormalised trees.
\end{defi}

Let us introduce a \emph{sharp partition of unity}:  given a positive constant $C_1$ (to be chosen later),
let $\chi$ and $\Psi$ be functions defined on $\RRR_{+}$, such that 
\begin{equation}
\label{eq:6.2}
\chi(x):=
\begin{cases}
1 &\text{for}\,\, x< \displaystyle{\frac{C_{1}}{4}} \e^{\frac{1}{\gotn+1}},\\ \vspace{-.4cm} \\
0 &\text{for}\,\, x\ge \displaystyle{\frac{C_{1}}{4}} \e^{\frac{1}{\gotn+1}},
\end{cases}
\qquad 
\Psi(x):=
\begin{cases}
1 &\text{for}\,\, x\ge \displaystyle{\frac{C_{1}}{4}} \e^{\frac{1}{\gotn+1}},\\ \vspace{-.4cm} \\
0 &\text{for}\,\, x < \displaystyle{\frac{C_{1}}{4}} \e^{\frac{1}{\gotn+1}}.
\end{cases}
\end{equation}
%

We associate with each node $v\in N(\vartheta)$  a \textit{node factor}
\begin{equation} 
\label{eq:6.3} 
F_{v} := \begin{cases}
- \e \, g_{p_{v}},  &  v \in V(\vartheta),  \\
\e \, f_{\nn_{v}},  & v \in E_{1}(\vartheta),  \\
\zeta, & v \in E_{0}(\vartheta), 
\end{cases}
\end{equation}
and with each line $\ell\in L(\vartheta)$ a \textit{propagator}
%
$G_{\ell} \equiv G^{[\tilde{n}_{\ell}]}(\oo \cdot \nn_{\ell}; \e, \zeta, c)$,
%
where the functions $G^{[\tilde{n}_{\ell}]}(\oo \cdot \nn_{\ell}; \e, \zeta, c)$ are defined as follows. 
For $\nn_{\ell}\ne\vzero$, set
\begin{subequations} \label{eq:6.4}
\begin{align}
\label{eq:6.4a}
G^{[0]}(\oo \cdot \nn_{\ell}; \e, \zeta, c) &:= \frac{\Psi(|\oo \cdot \nn_{\ell}|)}{\ii \oo \cdot \nn_{\ell} (1 + \ii\e \oo \cdot \nn_{\ell})},\\
\label{eq:6.4b}
G^{[1]}(\oo \cdot \nn_{\ell}; \e, \zeta, c)&:= \frac{\chi(|\oo \cdot \nn_{\ell}|)}{\ii \oo \cdot \nn_{\ell} (1 + \ii\e \oo \cdot \nn_{\ell}) - \mathcal{M}(\oo \cdot \nn_{\ell}; \e, \zeta, c)},
\end{align}
\end{subequations}
with 
\begin{subequations} \label{eq:6.5}
	\begin{align}
	\label{eq:6.5a}
	\mathcal{M}(\oo \cdot \nn_{\ell}; \e, \zeta, c) &:= \chi(|\oo \cdot \nn_{\ell}|) M(\oo \cdot \nn_{\ell}; \e, \zeta, c),\\
	\label{eq:6.5b}
	M(\oo \cdot \nn_{\ell}; \e, \zeta, c) &:= \sum_{T \in \gotR_{0}} \Val(T, \oo \cdot \nn_{\ell}; \e, \zeta, c),\\
	\label{eq:6.5c}
	\Val(T, \oo \cdot \nn;\e, \zeta, c) &:= \biggr(\prod_{\ell \in L(T)} G_{\ell}\biggl) \biggr(\prod_{v \in V(T)} F_{v}\biggl),
	\end{align}
\end{subequations}
where $\Val(T, \oo \cdot \nn_{\ell}; \e, \zeta, c)$ is called the \emph{value of the self-energy cluster} $T$.
For $\nn_{\ell} = \vzero$, we assign to $\ell$ the scale $0$ only, and set $G_{\ell}=G^{[0]}(0;\e,\zeta,c)=1$.	
%
Note that, with the sharp partition considered above, for any line $\ell\in L(\vartheta)$
the momentum $\nn_{\ell}$ identifies uniquely the scale $\tilde{n}_{\ell}$.

In order to simplify the notation we omit the dependence of the parameters $\e,\zeta, c$;
 hence, from now on,  we will write 
$G^{[0]}(\oo\cdot\nn)$, $G^{[1]}(\oo\cdot\nn),\, \mathcal{M}(\oo \cdot \nn),\, M(\oo \cdot \nn),\, \Val(T, \oo \cdot \nn)$.

\begin{rmk}
\label{rmk:6.6}
\emph{
Since only cluster on scale $0$ are considered, the product over the lines in \eqref{eq:6.7c} involves
only propagators on scale $0$, so that
the expressions of $\mathcal{M}(\oo \cdot \nn)$ in \eqref{eq:6.5a} and of $\Val(T,\oo\cdot\nn)$ in \eqref{eq:6.5c} are well-defined.
}
\end{rmk}

If 
$I_n(C_1,C_0)$ is defined as in \eqref{eq:2.3},
the condition $\e \in I_{n}(C_{1}, C_{0})$ can be rewritten as 
\begin{equation} \label{eq:6.6}
\frac{1}{C_{0}}\log\frac{1}{\e} < q_{n} < \frac{1}{C_{1}\e^{\frac{1}{\gotn+1}}} . 
\end{equation}

Let $\e$ satisfy \eqref{eq:6.6}. Consider a line $\ell \in L_{1}(\vartheta)$: if $|\nn_{\ell}|<q_{n}$, then also $|\nu_{\ell,2}|< q_{n}$,
so that, by Proposition \ref{prop:5.6} and by \eqref{eq:6.2} and \eqref{eq:6.6}, the line $\ell$ has to be on scale $0$.
\emph{Vice versa} if $\ell$ is on scale 1, then $|\nn|\ge q_{n}$.
%
Therefore, for $\e$ satisfying \eqref{eq:6.6},  we have three different possibilities for a line $\ell \in L_{1}(\vartheta)$:
	\vspace{-.2cm}
	\begin{enumerate}
\itemsep0em
	\item $|\nn_{\ell}|< q_{n}$ (this automatically implies scale $0$),
	\item $|\nn_{\ell}|\ge q_{n}$ and scale   $\tilde{n}_{\ell}=0$,
	\item $|\nn_{\ell}|\ge q_{n}$ and scale   $\tilde{n}_{\ell}=1$.
\end{enumerate}

In the first two cases, one has $G_{\ell}=G^{[0]}(\oo\cdot\nn_{\ell})$,
while in the last case one has $G_{\ell}=G^{[1]}(\oo\cdot\nn_{\ell})$.
We define $L_{<,0}(\vartheta),\, L_{\ge,0}(\vartheta),\, L_{\ge,1}(\vartheta)\,$ as follows:
\begin{subequations} \label{eq:6.7}
\begin{align}
L_{<,0}(\vartheta) & := \{ \ell \in L_{1}(\vartheta) : |\nn_{\ell}|< q_{n}\} , 
\label{eq:6.7a} \\
L_{\ge,0}(\vartheta) & := \{ \ell \in L_{1}(\vartheta) : |\nn_{\ell}| \ge q_{n} \,\,\text{and}\,\, \tilde{n}_{\ell}=0\} ,
\label{eq:6.7b} \\
L_{\ge,1}(\vartheta) & := \{ \ell \in L_{1}(\vartheta) : |\nn_{\ell}| \ge q_{n} \,\,\text{and}\,\, \tilde{n}_{\ell}=1\} .
\label{eq:6.7c}
\end{align}
\end{subequations}
By construction, one has $L_{1}(\vartheta)= L_{<,0}(\vartheta) \sqcup L_{\ge,0}(\vartheta) \sqcup L_{\ge,1}(\vartheta)$.

The \emph{value of the renormalised tree} $\vartheta$ is defined as 
\begin{equation}
\label{eq:6.8}
\Val(\vartheta)\equiv \Val(\vartheta; \e, \zeta, c):= \biggr(\prod_{\ell \in L(\vartheta)} G_{\ell}\biggl) \biggr(\prod_{v \in V(\vartheta)} F_{v}\biggl) .
\end{equation}	

Finally set 
\begin{equation}
\label{eq:6.9}
u^{[k]}_{\nn} := \sum_{\vartheta \in \gotT_{k,\nn}} \Val(\vartheta), \qquad \nn \in \ZZZ^d ,
\end{equation}
where $\gotT_{k, \nn}$ denotes the set of all renormalised trees of order $k$ and momentum $\nn$ associated with the root line,
and define the \emph{renormalised series} as 
\begin{equation}
\label{eq:6.10}
\bar{u}(\pps) \equiv \bar{u}(\pps; \e, \zeta, c):= \sum_{\nn \in Z^{2}_{*}}
{\rm e}^{\ii \nn \cdot \pps} \bar{u}_{\nn}, \quad \bar{u}_{\nn} :=\sum_{k=1}^{\infty} u^{[k]}_{\nn} , 
\vspace{-.2cm}
\end{equation}	
where, once more, the dependence on $\e, c$ and $\zeta$ of the coefficients $u_{\nn}^{[k]}$ is omitted.

\zerarcounters
\section{Bounds on the values of the self-energy clusters}
\label{sec:7}

Define
\begin{equation}
\label{eq:7.1}
\eta:=\max\{\e, |\zeta|\}.
\end{equation}
%
\begin{lemma}
	\label{lem:7.1}
	For any self-energy cluster $T$ one has
	\begin{equation} \nonumber
	|\Val(T,\oo\cdot\nn)|\le \Gamma \rho \, \bar{C}^{k_{T}} \e \eta^{\frac{\gotn}{\gotn+1} (k_{T} -1)},
	\end{equation}
	where
	\begin{equation}
	\label{eq:7.2}
	\bar{C}:= \rho^{-1} \max\Big\{\frac{4\Phi}{C_{1}}, \frac{4\Gamma}{C_{1}}, 1\Big\},
	\end{equation}
	with $\rho,\, \Gamma$ as in \eqref{eq:2.2} and $\Phi$ as in \eqref{eq:2.1}.
\end{lemma}

\begin{proof}
	We recall the definition of the value of a self-energy cluster (see remark \ref{rmk:6.6}), that is
	\[
	\Val(T,\oo\cdot\nn)= \Big(\prod_{v \in V(T)} F_{v}\Big) \Big(\prod_{\ell \in L(T)} G^{[0]}(\oo\cdot\nn_{\ell}) \Big).
	\] 
	Since every line in $T$ is on scale $0$, we bound the propagators as 
	\[
	|G^{[0]}(\oo\cdot\nn_{\ell})| \le \frac{1}{|\oo\cdot\nn_{\ell}|} \le \frac{4}{C_{1}} \e^{-\frac{1}{\gotn+1}}, 
	\]
	if $\nn_{\ell} \ne 0$, while one has
	%
	$G^{[0]}(\oo\cdot\nn_{\ell}) =1$
	%
	if $\nn_{\ell}= \vzero$. 
	Then, by using \eqref{eq:2.1} and \eqref{eq:2.2} to bound the node factors, we have 
	\[
	\begin{split}
	|\Val(T,\oo\cdot\nn)|&\le \Gamma \rho^{-(k_{T}-1)}  \Big(\frac{4\Phi}{C_{1}}\Big)^{|E_{1}(T)|} 
	\Big(\frac{4\Gamma}{C_{1}}\Big)^{|V(T)|-1} 
	\e^{\frac{\gotn}{\gotn+1}( |V(T)| - 1 + |E_{1}(T)|) + 1} \, |\zeta|^{|E_{0}(T)|} \\
	&\le \Gamma \rho \, \bar{C}^{k_{T}} \e \,\eta^{\frac{\gotn}{\gotn+1}(|V(T)| - 1 + |E_{1}(T)|+ |E_{0}(T)|) + \frac{1}{\gotn+1} |E_{0}(T)|}
	\le \Gamma \rho \, \bar{C}^{k_{T}} \e \,\eta^{\frac{\gotn}{\gotn+1}  (k_{T}-1)},
	\end{split}
	\]
	where we have used that $|V(T)|\ge 1$ by construction and defined $\bar{C}$ as in \eqref{eq:7.2}.
\end{proof}

\begin{lemma}
	\label{lem:7.2} 
	Define  $\MM(\oo \cdot \nn)$ as in  \eqref{eq:6.5a}. Then, for $\eta$ small enough, one has 
	\begin{equation} \nonumber
	|\MM(\oo \cdot \nn)| \ge A \e \eta^{\gotn-1},
	\end{equation}
	with  $A$ positive constant  depending on $\Gamma$, $\rho$ and $ \Phi$, 
	with $\Gamma, \rho$ as in \eqref{eq:2.2} and $\Phi$ as in \eqref{eq:2.1}.
\end{lemma} 

\begin{proof}
		A cluster $T$ must contain at least $\gotn$ nodes, i.e. $k_{T}\ge \gotn$. Indeed let us consider a tree with order $k\ge \gotn +1$, in which the root line, $\ell_{0}$, exits a node $v_{0}\in V(\vartheta)$. By construction $p_{v_{0}}\ge \gotn$, therefore a cluster $T$, if exists, must contain at least $\gotn -1$ lines on scale $0$ entering $v_{0}$  and hence $\gotn -1$ nodes besides $v_{0}$.  
		
		Moreover , if a self-energy cluster has only $\gotn$ nodes, then $\gotn -1$ of such nodes
		are in $E(\vartheta)$ and the external lines of the cluster exit/enter the same node -- see Figure \ref{figselfenergy}.
		
		\begin{figure}[h]
			\centering
			\includegraphics[scale=0.50]%
			{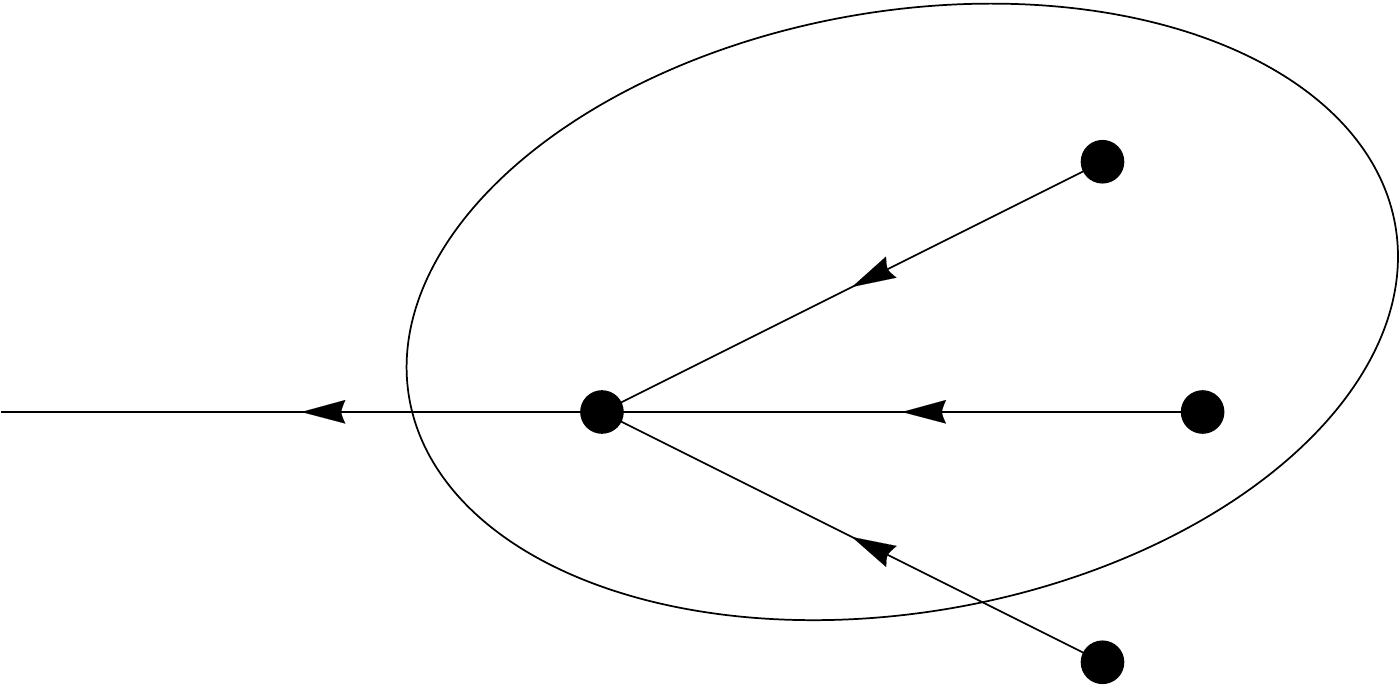}
			\caption{A self-energy with $3$ nodes, for $\gotn = 3$.}  
			\label{figselfenergy}
		\end{figure}

Denote by $\MM_{\gotn}(\oo\cdot\nn)$ the terms of $\MM(\oo\cdot\nn)$ corresponding to  self-energy clusters with $\gotn$ nodes and
by $\Delta\MM(\oo\cdot\nn)$ the sum of all the other terms contributing to $\MM(\oo\cdot\nn)$.
Notice that $\MM_{\gotn}(\oo\cdot\nn)=\MM_{\gotn}(0)$, i.e.~$\MM_{\gotn}(\oo\cdot\nn)$ does not depend on $\nn$ and is real. 
		Hence we have
		\begin{equation}
		\label{eq:7.3}
		|\MM_{\gotn}(\oo\cdot\nn)| \ge A_{0} \e \bigl[(\bar u^{[1]} +\zeta)^{\gotn-1}\bigr]_{\vzero}, \qquad
		\bar u^{[1]}(\psi) :=  \e \sum_{\nn \in \ZZZ^2_*} {\rm e}^{i\nn\cdot\pps} f_{\nn} \, G^{[0]}(\oo\cdot\nn) ,
		\vspace{-.2cm}
		\end{equation}
		for a suitable positive constant $A_{0}$ depending on $\Phi, \Gamma, \rho, \xi$. 
		
		If $\zeta= o(\e)$, then one has
		%
		$\bigl[(u^{[1]} +\zeta)^{\gotn-1}\bigr]_{\vzero} 
		\ge A_{1} \e^{\gotn-1}$,
		%
		for a suitable positive constant $A_{1}$ depending on $\Phi, \xi$;
		analogously, if $\e=o(\zeta)$, then
		%
		$\bigl[(u^{[1]} +\zeta)^{\gotn-1}\bigr]_{\vzero} 
		\ge A_{2}\zeta^{\gotn-1}$,
		%
		for a suitable positive constant $A_{2}$, 
		finally, if $\zeta= a \e + o(\e)$, with $a \ne 0$, then, if one sets $\bar u^{[1]} = \e X^{[1]}$, one has
		\[
		\begin{split}
		\bigl[(u^{[1]} +\zeta)^{\gotn-1}\bigr]_{\vzero} &= \bigl[(u^{[1]} + a\e +o(\e))^{\gotn-1}\bigr]_{\vzero} 
		\ge [(X^{[1]} +a)^{\gotn-1}]_{\vzero}\, \e^{\gotn-1} +o(\e^{\gotn-1}) 
		\ge A_{3}\, \e^{\gotn-1},
		\end{split}
		\]
		with $A_{3}$ a positive constant. 
		In conclusion, one obtains
		\begin{equation}
		\label{eq:7.4}
		|\MM_{\gotn}(\oo\cdot\nn)| \ge \tilde{A}_{0} \e \eta^{\gotn-1},
		\end{equation}
		for a suitable positive constant $\tilde{A}_{0}$ depending on $A_{0}$, $A_{1}$, $A_{2}$, $A_{3}$. 
		
		Using \eqref{eq:7.4} and Lemma \ref{lem:7.1} to bound $\Delta\MM(\oo\cdot\nn)$, one finds, for $\eta$ small enough,
		%
		%
		\[
		\begin{split}
		|\MM(\oo\cdot\nn)| &\ge |\MM_{\gotn}(\oo\cdot\nn)| - |\Delta\MM(\oo\cdot\nn)| 
		\ge \tilde{a}_{0} \e \eta^{\gotn-1} - \tilde{a}_{1} \e \eta^{\frac{\gotn^{2}}{\gotn+1}} 
		\ge A \e \eta^{\gotn -1} 
		\end{split}
		\]
		for suitable positive constants $\tilde{a}_{0}$ and $\tilde{a}_{1}$ depending on $\tilde{A}_{0}$, $\rho$
		and $\bar{C}$ and $A=\tilde{a}_{0}/2$, provided $\eta$ is small enough. 
\end{proof}

\zerarcounters
\section{The range equation}
\label{sec:8}

The goal of the present section is to prove that the series  $\bar{u}(\pps; \e, \zeta, c)$
in \eqref{eq:6.10} converges and solves \eqref{eq:4.2}, i.e. 
\begin{equation}
\label{eq:8.1}
\ii \oo \cdot \nn (1 + \ii \oo \cdot \nn)\bar{u}_{\nn}=  \e \, [f - g(c + \zeta + \bar{u})]_{\nn}, 
\end{equation}	
for all $\nn\ne0$ and for all $\zeta$ small enough, independently of $\e$.

To prove the convergence of the series, we have to provide a bound for the coefficients $u_{\nn}^{[k]}$ in \eqref{eq:6.9} and hence
we need an estimate on the value \eqref{eq:6.8} of each renormalised tree $\vartheta$.
The latter will be the main object of Subsection \ref{sec:8.1}.
At the end, we will prove that the renormalised series solves the equation \eqref{eq:8.1}: this will be discussed in Subsection \ref{sec:8.2}.
Finally, in Section \ref{sec:9} we will fix $\zeta=\zeta(\e)$ as a function of $\e$ which goes to zero as $\e$ goes
to zero, in order to make the bifurcation equation to be satisfied.

\subsection{Bounds on the values of the renormalised trees}
\label{sec:8.1}


Let $C$, $\tilde{C}$ and $\tilde{C}_{1}$ be defined as follows:
\begin{subequations} \label{eq:8.2}
\begin{align}
\label{eq:8.2a}
C & := \rho^{-1} \max\Big\{\frac{4\Phi}{C_{1}},\frac{4\Phi}{A},1,\frac{4\Gamma}{A},\frac{4\Gamma}{C_{1}}\Big\}, \\
\label{eq:8.2b}
\tilde{C} & :=\max\Big\{\frac{\Phi}{\Gamma}, \frac{C_{1}\Phi}{\Gamma A}, \frac{C_{1}}{4\Gamma} ,\frac{C_{1}}{ A},1\Big\},\\
\label{eq:8.2c}
\tilde{C}_{1} & :=\max\Big\{\frac{\Phi A}{C_{1}\Gamma}, \frac{\Phi}{\Gamma }, \frac{A}{4\Gamma} ,\frac{A}{C_{1}},1\Big\},
\end{align} 
\end{subequations}
with the positive constants $\rho, \Gamma,\Phi$ defined as in \eqref{eq:2.2} and \eqref{eq:2.1}
and $A$ defined as in Lemma \ref{lem:7.2}. Note that $\rho\,C,\tilde{C}, \tilde{C}_{1}\ge 1$.

\begin{lemma}
	\label{lem:8.1}
	Let $\vartheta $ be a renormalised tree of order $k=1$ and let  $\ell_{0}$ denote the root line of $\vartheta$.
	Then 
	one has
	\begin{equation}
	\label{eq:8.3}
	|\Val(\vartheta)|\le \rho \, C  \times
	\begin{cases}
	\e^{\frac{\gotn}{\gotn+1}} {\rm e}^{-\xi|\nn|}, \qquad &\text{if $\ell_{0} \in L_{<,0}(\vartheta) \sqcup L_{\ge,0}(\vartheta)$},\\
	\eta^{-\gotn+1} {\rm e}^{-\xi|\nn|}, \qquad &\text{if $\ell_{0} \in L_{\ge,1}(\vartheta)$},\\
	|\zeta|, \qquad &\text{if $\ell_{0} \in L_{0}(\vartheta)$},
	\end{cases}
	\end{equation}
	with $\rho$, $\eta$ and $C$ as in \eqref{eq:2.2}, \eqref{eq:7.1} and \eqref{eq:8.2a}, respectively.
\end{lemma}

\begin{proof}
If $\ell_{0} \in L_{<,0}(\vartheta)\sqcup L_{\ge,0}(\vartheta)$, then the value of the tree is
$\Val(\vartheta) = \e f_{\nn}\,G^{[0]}(\oo\cdot\nn)$.
%
If we bound $|f_{\nn}|$ according to \eqref{eq:2.1} and $|G^{[0]}(\oo\cdot\nn)|$ by using the sharp partition \eqref{eq:6.2},
i.e.~$|G^{[0]}(\oo\cdot\nn)|\le 4 C_{1}^{-1} \e^{-\frac{1}{\gotn+1}}$, we obtain the first bound in \eqref{eq:8.3}.  
	
If  $\ell_{0} \in L_{\ge,1}(\vartheta)$, we have
$\Val(\vartheta) = \e f_{\nn}\,G^{[1]}(\oo\cdot\nn)$
and, by using Lemma \ref{lem:7.2} in order to bound $|G^{[1]}(\oo\cdot\nn)| \le 4 (A\e)^{-1}\eta^{-\gotn+1}$, 
we obtain the second bound in \eqref{eq:8.3}.
	
Finally, if $\ell_{0}\in L_{0}(\vartheta)$, then $\Val(\vartheta)=\zeta$. 
\end{proof}

\begin{lemma}
	\label{lem:8.2}
	Let $\vartheta \in \gotT_{k, \nn}$ be a renormalised tree of order $k\ge \gotn+1$ and momentum $\nn\ne\vzero$ associated
	with the root line $\ell_{0}$. Let $C_{0}$ be as in \eqref{eq:2.4} and, for any fixed $C_1$ take $\e \in I_{n}(C_{1},C_{0})$
	for some $n \ge N$, where $N$ satisfies \eqref{eq:2.6b}. Then one has
	%
	\begin{equation}
	\label{eq:8.4}
	\Val(\vartheta) =  \overline{\Val}(\vartheta) \prod_{v \in E_{1}(\vartheta)} {\rm e}^{-\frac{\xi}{2}|\nn_{v}|},
	\end{equation}
	where
	\begin{equation}
	\label{eq:8.5}
	|\overline{\Val}(\vartheta)|\le \rho \, C^{k} \eta^{\frac{k}{\gotn(\gotn+1)} + \frac{\gotn^{2}-1}{\gotn}} \times 
	\begin{cases}
	{\rm e}^{-\frac{\xi}{4}|\nn|}, \quad &\text{if $\ell_{0}\in L_{<,0}(\vartheta)$,}\\
	{\rm e}^{-\frac{\xi}{4}q_n}, \quad &\text{if $\ell_{0}\in L_{\ge,0}(\vartheta)$,}\\
	\eta^{-\gotn+1} {\rm e}^{-\frac{\xi}{4}q_n}, \quad &\text{if $\ell_{0} \in L_{\ge,1}(\vartheta)$,}\\
	\end{cases}
	\end{equation}
	with $\rho$, $\eta$ and $C$ as in \eqref{eq:2.2}, \eqref{eq:7.1} and \eqref{eq:8.2a}, respectively.
\end{lemma}

\begin{proof}
Denote by $v_{0}$ the first node of $\vartheta$, that is the node the root line $\ell_0$ exits.
The renormalised tree $\vartheta$ has the following structure:
\begin{itemize}
\itemsep0em
\item $\vartheta_{1} \in \gotT_{k_{1},\nn_{\ell_{1}}},\dots,\vartheta_{m}\in \gotT_{k_{m},\nn_{\ell_{m}}}$
enter $v_{0}$ and the root lines $\ell_{1},\dots,\ell_{m}$ are such that $|\nn_{\ell_{j}}|<q_{n}$
for all $j=1,\dots,m$, so that $\ell_{1},\dots,\ell_{m} \in L_{<,0}(\vartheta)$;
\item $\vartheta'_{1} \in \gotT_{k'_{1},\nn_{\ell'_{1}}},\dots,\vartheta'_{p} \in \gotT_{k'_{p},\nn_{\ell'_{p}}}$
enter $v_{0}$ and the root lines $\ell'_{1},\dots,\ell'_{p}$ are such that $|\nn_{\ell'_{j}}|\ge q_{n}$
and $\tilde{n}_{\ell'_{j}}=0$ for all $j=1,\dots,p$, so that $\ell'_{1},\dots,\ell'_{p} \in L_{\ge,0}(\vartheta)$;
\item $\vartheta''_{1} \in \gotT_{k''_{1},\nn_{\ell''_{1}}},\dots,\vartheta''_{l} \in \gotT_{k''_{l},\nn_{\ell''_{l}}}$
enter $v_{0}$ and the root lines $\ell''_{1},\dots,\ell''_{l}$ are such that $|\nn_{\ell''_{j}}|\ge q_{n}$
and $\tilde{n}_{\ell''_{j}}=1$ for all $j=1,\dots,l$, so that $\ell''_{1},\dots,\ell''_{l} \in L_{\ge,1}(\vartheta)$;
\item the lines $\tilde{\ell}_{1},\dots, \tilde{\ell_{r}}$, entering $v_{0}$, exit the end nodes
$\tilde{v}_{1}, \dots, \tilde{v}_{r} \in E_{1}(\vartheta)$ respectively and are such that $\tilde{n}_{\tilde{\ell}_{j}} =0$
for all $j=1,\dots,r$, so that  $\tilde{\ell}_{j}\in L_{<,0}(\vartheta)\sqcup L_{\ge, 0}(\vartheta)$ for all $j=1,\dots,r$;
\item the lines $\tilde{\ell'}_{1},\dots, \tilde{\ell'_{s}}$, entering $v_{0}$, exit the end nodes
$\tilde{v'}_{1}, \dots, \tilde{v'}_{s} \in E_{1}(\vartheta)$ respectively and $\tilde{\ell'}_{j} \in L_{\ge,1}$ for all $j=1,\dots,s$;
\item the lines $\bar{\ell}_{1},\dots, \bar{\ell}_{u}$, entering $v_{0}$, 
exit the end nodes $\bar{v}_{1}, \dots, \bar{v}_{u} \in E_{0}(\vartheta)$;
\item $m,p,l,r,s,u \ge 0$.
\end{itemize}
According to this construction, we have the following constraints:
\begin{itemize}
\itemsep0em
\item $k=k(\vartheta)= \sum_{j=1}^{m} k(\vartheta_{j}) + \sum_{j=1}^{p} k(\vartheta'_{j}) + \sum_{j=1}^{l} k(\vartheta''_{j}) + r + s+u+1$;
\item $m + p+ l+ r+ s+ u \ge \gotn$;
\item $\nn= \sum_{j=1}^{m} \nn_{\ell_{j}} + \sum_{j=1}^{p} \nn_{\ell'_{j}} +
\sum_{j=1}^{l} \nn_{\ell''_{j}} + \sum_{j=1}^{r} \nn_{\tilde{\ell}_{j}} + \sum_{j=1}^{s} \nn_{\tilde{\ell'}_{j}};$
\item $\nn\ne\vzero$.
\end{itemize}
	
In the expression of $\Val(\vartheta)$, after extracting a factor $\prod_{v\in E_{1}(\vartheta)} {\rm e}^{-\xi|\nn_{v}|/2}$,
we verify by induction on the order of the tree the inequalities in \eqref{eq:8.5}. 
In particular we will use \eqref{eq:8.5} to bound the values of the subtrees of order $k'$,
with $1<k'<k$ and \eqref{eq:8.3} with $\xi$ replaced with $\frac{\xi}{2}$
to bound the values of the subtrees of order $1$ formed by an end node
in $E_{1}(\vartheta)$ and the  corresponding exiting line.
Finally we use the last inequality in \eqref{eq:8.3} to estimate
the values of the subtrees of order $1$ formed by end nodes in $E_{0}(\vartheta)$ and the corresponding exiting lines.

\subsubsection*{Case 1: $\boldsymbol{\ell_{0} \in L_{<,0}(\vartheta)}$}

We start by analysing the case in which the root line has momentum $\nn$ such that $|\nn|<q_{n}$. In this case we have:
\[
\begin{split}
|\overline{\Val}(\vartheta)|&\le \frac{4\Gamma}{C_{1}} \rho^{-(m+p+l+r+s)}\rho^{m+p+l+r+s}C^{k-1}
\eta^{\frac{k-r-s-u-1}{\gotn(\gotn+1)} + \frac{\gotn^{2}-1}{\gotn}(m+p+l)} \e^{\frac{\gotn}{\gotn+1}(r+1)} \times \\
& \qquad \times \eta^{(-\gotn+1)(l+s)} |\zeta|^{u}  {\rm e}^{-\frac{\xi}{4} \sum_{j=1}^{m} |\nn_{\ell_{j}}|} 
{\rm e}^{-\frac{\xi}{4}q_{n}(p+l)} {\rm e}^{-\frac{\xi}{2} \big(\sum_{j=1}^{r} |\nn_{\tilde{\ell}_{j}}| +\sum_{j=1}^{s} |\nn_{\tilde{\ell}'_{j}}|\big)}\\
&\le \rho \, C^{k} \tilde{C}^{-1}
\eta^{\frac{k}{\gotn(\gotn+1)} + \frac{\gotn^{2}-1}{\gotn}}\eta^{\Delta_{0}(m,p,l,r,s,u)} 
{\rm e}^{-\frac{\xi}{4} \sum_{j=1}^{m} |\nn_{\ell_{j}}|} {\rm e}^{-\frac{\xi}{4}q_{n}(p+l)} 
{\rm e}^{-\frac{\xi}{2} \big(\sum_{j=1}^{r} |\nn_{\tilde{\ell}_{j}}| +\sum_{j=1}^{s} |\nn_{\tilde{\ell}'_{j}}|\big)} \!\! ,
\end{split}
\]
where the constants $C$ and $\tilde{C}$ are as in \eqref{eq:8.2a} and in \eqref{eq:8.2b}, respectively, and we have defined
\[
\begin{split}
& \Delta_{0}(m,p,l,r,s,u) \\
& \qquad := \frac{(m \!+\! p)(\gotn^{3} \!+\! \gotn^{2} \!-\! \gotn-1) \!+\! (l \!+\! r)(\gotn^{2} \!-\! 1) \!+\! u ( \gotn^{2} \!+\! \gotn \!-\! 1)
- s(\gotn^{3}  - \gotn \!+\! 1) \!+\!  \gotn^{3}  \!-\!  \gotn}{\gotn(\gotn +1)}
\end{split}
\]
%

To obtain the first bound in \eqref{eq:8.5}, we have to prove that
\begin{equation}
\label{eq:8.6}
\tilde{C}^{-1}
\eta^{\Delta_{0}(m,p,l,r,s,u)} 
{\rm e}^{-\frac{\xi}{4}q_{n}(p+l)} {\rm e}^{-\frac{\xi}{4} \sum_{j=1}^{m} |\nn_{\ell_{j}}| -\frac{\xi}{2} 
\big(\sum_{j=1}^{r} |\nn_{\tilde{\ell}_{j}}| +\sum_{j=1}^{s} |\nn_{\tilde{\ell}'_{j}}|\big)} \le {\rm e}^{-\frac{\xi}{4}|\nn|}.
\end{equation}
We distinguish among three different cases: $p+l=1$, $p+l\ge2$ and $p+l=0$.
%

\subsubsection*{Case 1.1: $\boldsymbol{p+l=1}$}

Since ${\rm e}^{-\frac{\xi}{4}q_{n}(p+l)}={\rm e}^{-\frac{\xi}{4}q_{n}}<{\rm e}^{-\frac{\xi}{4}|\nn|}$, we have to prove 
\begin{equation}
\label{eq:8.7}
\tilde{C}^{-1}\eta^{\Delta_{0}(m,p,l,r,s,u)} 
\times {\rm e}^{-\frac{\xi}{4} \sum_{j=1}^{m} |\nn_{\ell_{j}}|} 
\times {\rm e}^{-\frac{\xi}{2} \big(\sum_{j=1}^{r} |\nn_{\tilde{\ell}_{j}}| +\sum_{j=1}^{s} |\nn_{\tilde{\ell}'_{j}}|\big)} \le 1.
\end{equation}
%

\subsubsection*{Case 1.1.1: $\boldsymbol{p=1, \; l=0}$}
	
When $s=0$, the bound \eqref{eq:8.7} is obviously satisfied because
\begin{equation} \nonumber
\Delta_{0}(m,1,0,r,0,u)\ge \frac{(\gotn^{2}-1)(m+r+u) + \gotn^{2} -1}{\gotn(\gotn+1)} 
\ge \frac{ \gotn(\gotn^{2} -1)}{\gotn(\gotn+1)}=  \gotn -1,
\end{equation}
The bound \eqref{eq:8.7} is satisfied for $s\ge1$ as well: first, we observe that
\begin{equation} \nonumber
\Delta_{0}(m,1,0,r,s,u)\ge -\frac{\gotn^{3}-\gotn+1}{\gotn(\gotn+1)}s, 
\end{equation}
so that we have
	\[
	\begin{split}
	& \tilde{C}^{-1} \eta^{\Delta_{0}(m,1,0,r,s,u)} {\rm e}^{-\frac{\xi}{4} \sum_{j=1}^{m} |\nn_{\ell_{j}}|} 
	{\rm e}^{-\frac{\xi}{2} \big(\sum_{j=1}^{r} |\nn_{\tilde{\ell}_{j}}| +\sum_{j=1}^{s} |\nn_{\tilde{\ell}'_{j}}|\big)}
	\le  \eta^{-\frac{\gotn^{3}-\gotn+1}{\gotn(\gotn+1)}s}  {\rm e}^{-\frac{\xi}{2} \sum_{j=1}^{s} |\nn_{\tilde{\ell}'_{j}}|} ,
	\end{split}
	\]
	then, using that $|\nn_{\tilde{\ell}'_{j}}|\ge q_{n}\ge C_{0}^{-1}\log \e^{-1}$ for all $j=1,\dots, s$,  we deduce 
	\[
	\eta^{-\frac{\gotn^{3}-\gotn+1}{\gotn(\gotn+1)}s}  
	{\rm e}^{-\frac{\xi}{2} \sum_{j=1}^{s} |\nn_{\tilde{\ell}'_{j}}|}
	\le \Big(\eta^{-\frac{\gotn^{3}-\gotn+1}{\gotn(\gotn+1)}} \e^{\frac{\xi}{2C_{0}}}\Big)^{s}  
	\le \Big(\eta^{-\frac{\gotn^{3}-\gotn+1}{\gotn(\gotn+1)}} \eta^{\frac{\xi}{2C_{0}}}\Big)^{s} \le 1 ,
	\] 
	where the last inequality follows by taking
	%
\begin{equation}
\label{eq:8.8}	
C_{0} \le \frac{\gotn(\gotn+1)}{2(\gotn^{3}-\gotn+1)}\xi .
\end{equation}
%
	
\subsubsection*{Case 1.1.2: $\boldsymbol{p=0, \; l=1}$}

If $s=0$, one has $m+r+u\ge\gotn-1$ and hence
	\[
	\Delta_{0}(m,0,1,r,0,u)\ge \frac{(\gotn^{2}-1)(m+r+u) -\gotn^{3}+\gotn^{2}+\gotn-1}{\gotn(\gotn+1)} \ge 0. 
	\]
If $s\ge 1$, $\Delta_{0}(m,0,1,r,s,u)$ may be negative, but, as in the case 1.1.1, we can use the exponential decay associated with the
lines $\tilde \ell'_1,\ldots,\tilde{\ell}'_s$. Since we have
	\[
	\Delta_{0}(m,0,1,r,s,u)\ge - \frac{\gotn^{3} + \gotn^{2} - \gotn}{\gotn(\gotn+1)}s = 
	- \frac{\gotn^{2} + \gotn - 1}{\gotn+1}s, 
	\]
if we require
%
\begin{equation}
\label{eq:8.9}
C_{0}\le \frac{\gotn+1}{2(\gotn^{2} + \gotn -1)}\xi ,
\end{equation}
the bound \eqref{eq:8.7} follows, noting that
	\[
	\begin{split}
	&\tilde{C}^{-1}\, \eta^{\Delta_{0}(m,0,1,r,s,u)} {\rm e}^{-\frac{\xi}{4} \sum_{j=1}^{m} |\nn_{\ell_{j}}| 
		-\frac{\xi}{2} \big(\sum_{j=1}^{r} |\nn_{\tilde{\ell}_{j}}| +\sum_{j=1}^{s} |\nn_{\tilde{\ell}'_{j}}|\big)} 
		\le  \Big(\eta^{\frac{-\gotn^{2}-\gotn + 1}{\gotn+1}}\e^{\frac{\xi}{2C_{0}}}\Big)^{s} \le1.
	\end{split}
	\]
	%
	%
	
\subsubsection*{Case 1.2: $\boldsymbol{p+l\ge 2}$}

One has
	\[
	\Delta_{0}(m,p,l,r,s,u) \ge -\frac{\gotn^{3}-\gotn+1}{\gotn(\gotn+1)}s - \frac{(\gotn-2)(\gotn-1)}{\gotn}.
	\] 
By splitting 
	%
	${\rm e}^{-\frac{\xi}{4}q_{n}(p+l)}= {\rm e}^{-\frac{\xi}{4}q_{n}(p+l-1)} {\rm e}^{-\frac{\xi}{4}q_{n}}$,
	%
we use the factor ${\rm e}^{-\frac{\xi}{4}q_{n}}$ to cancel the factor ${\rm e}^{-\frac{\xi}{4}|\nn|}$ in right hand side of
\eqref{eq:8.6}, so that we have to prove that
\begin{equation}
\label{eq:8.10}	
\tilde{C}^{-1} \eta^{\Delta_{0}(m,p,l,r,s,u)}
{\rm e}^{-\frac{\xi}{4} \sum_{j=1}^{m} |\nn_{\ell_{j}}|} {\rm e}^{-\frac{\xi}{2} 
\big(\sum_{j=1}^{r} |\nn_{\tilde{\ell}_{j}}| +\sum_{j=1}^{s} |\nn_{\tilde{\ell}'_{j}}|\big)} {\rm e}^{-\frac{\xi}{4}q_{n}(p+l-1)} \le 1. 
\end{equation}
By using that $q_{n}\ge C_{0}^{-1}\log \e^{-1}$, if we require
\begin{equation}
\label{eq:8.11}
\begin{cases}
C_{0} \le \displaystyle{ \frac{\gotn(\gotn+1)}{2(\gotn^{3}-\gotn +1)}\xi} , \quad &\text{if\, $\gotn=3$},\\ \vspace{-.4cm} \\
C_{0} \le \displaystyle{ \frac{\gotn}{4(\gotn-2)(\gotn-1)}\xi } , \quad &\text{if $\gotn \ge 5$ and $\gotn$ odd},
\end{cases} 
\end{equation}
we obtain that the left hand side of \eqref{eq:8.10} is less than
\[
\begin{split}
\eta^{-\frac{(\gotn-2)(\gotn-1)}{\gotn} + \frac{\xi}{4C_{0}}}
\Big(\eta^{-\frac{\gotn^{3}-\gotn+1}{\gotn(\gotn+1)} + \frac{\xi}{2C_{0}}}\Big)^{s} \le 1.\\ 
\end{split}
\]
%

\subsubsection*{Interlude}

We compare the conditions obtained so far on $C_{0}$ (see \eqref{eq:8.8}, \eqref{eq:8.9} and \eqref{eq:8.11}):
	%
	%
\vspace{-.1cm}
\begin{itemize}
\itemsep0em
\item if $\gotn=3$, one has $\displaystyle{\frac{\gotn+1}{2(\gotn^{2} + \gotn - 1)} < \frac{\gotn(\gotn+1)}{2(\gotn^{3}-\gotn +1)}}$;
\item if $\gotn=5$, one has $\displaystyle{\frac{\gotn+1}{2(\gotn^{2} + \gotn - 1)} < \frac{\gotn}{4(\gotn-2)(\gotn-1)}}$;
\item if $\gotn \ge 7$, one has $\displaystyle{\frac{\gotn}{4(\gotn-2)(\gotn-1)} < \frac{\gotn+1}{2(\gotn^{2} + \gotn - 1)}}$.
\end{itemize}
Then, to summarize, the condition on $C_{0}$ becomes
\begin{equation}
\label{eq:8.12}
\begin{cases}
C_{0} \le \displaystyle{ \frac{\gotn+1}{2(\gotn^{2} + \gotn - 1)}\xi,} \quad &\text{if $\gotn = 3, 5$},\\ \vspace{-.4cm} 
C_{0} \le \displaystyle{ \frac{\gotn}{4(\gotn-2)(\gotn-1)}\xi, }  \quad &\text{if $\gotn \ge 7$ and $\gotn$ odd}.
\end{cases} 
\end{equation}
%

\subsubsection*{Case 1.3: $\boldsymbol{p+l=0}$}

The conservation law \eqref{eq:6.1} gives
%
${\rm e}^{-\frac{\xi}{4} \big(\sum_{j=1}^{m} |\nn_{\ell_{j}}|+\sum_{j=1}^{r} |\nn_{\tilde{\ell}_{j}}| +
\sum_{j=1}^{s} |\nn_{\tilde{\ell}'_{j}}|\big)} \le {\rm e}^{-\frac{\xi}{4}|\nn|}$,
%
so we have to prove that
\begin{equation}
\label{eq:8.13}
\tilde{C}^{-1}\eta^{\Delta_{0}(m,0,0,r,s,u)} {\rm e}^{-\frac{\xi}{4}
\big(\sum_{j=1}^{r} |\nn_{\tilde{\ell}_{j}}| +\sum_{j=1}^{s} |\nn_{\tilde{\ell}'_{j}}|\big)} \le 1.
\end{equation}
	If $s=0$, that is $m+r+u\ge\gotn$, the thesis follows trivially since $\Delta_{0}(m,0,0,r,0,u) \ge0$. 
	If $s\ge1$, since one has
	\[
	\Delta_{0}(m,0,0,r,s,u) \ge -\frac{\gotn^{2}+ \gotn-1}{\gotn+1}s
	\]
	one has to require 
	%
	$\eta^{-\frac{\gotn^{2}+\gotn-1}{\gotn+1}} \e^{\frac{\xi}{4C_{0}}} \le 1$,
	%
	that is 
	%
\begin{equation}
\label{eq:8.14} 
C_{0} \le \frac{\gotn+1}{4(\gotn^{2} +\gotn -1)}\xi.
\end{equation}
	%

\subsubsection*{Case 2: $\boldsymbol{\ell_{0} \in L_{\ge,0}(\vartheta)}$}

	We consider now the case in which the root line has momentum $\nn$ such that $|\nn|\ge q_{n}$ and is on scale $0$.
	We can estimate
	\[
	\begin{split}
	|\overline{\Val}(\vartheta)|&\le C^{k-1}\frac{4\Gamma}{C_{1}} \rho^{-(m+p+l+r+s+u)}
	\rho^{m+p+l+r+s+u} \eta^{\frac{k-r-s-u-1}{\gotn(\gotn+1)} + 
	\frac{\gotn^{2}-1}{\gotn}(m+p+l)} \e^{\frac{\gotn}{\gotn+1}(r+1)}  \zeta^{u} \times \\
	&\times \eta^{(-\gotn+1)(l+s)}  {\rm e}^{-\frac{\xi}{4} \sum_{j=1}^{m} |\nn_{\ell_{j}}|} 
	{\rm e}^{-\frac{\xi}{4}q_{n}(p+l)} {\rm e}^{-\frac{\xi}{2} \big(\sum_{j=1}^{r} |\nn_{\tilde{\ell}_{j}}| +
	\sum_{j=1}^{s} |\nn_{\tilde{\ell}'_{j}}|\big)}\\
	& \hskip-.4cm 
	\le \rho \, C^{k} \tilde{C}^{-1} \eta^{\frac{k}{\gotn(\gotn+1)} + \frac{\gotn^{2}-1}{\gotn}}\eta^{\Delta_{0}(m,p,l,r,s,u)} 
	{\rm e}^{-\frac{\xi}{4} \sum_{j=1}^{m} |\nn_{\ell_{j}}| - \frac{\xi}{2} 
	\big(\sum_{j=1}^{r} |\nn_{\tilde{\ell}_{j}}| +\sum_{j=1}^{s} |\nn_{\tilde{\ell}'_{j}}|\big)}  {\rm e}^{-\frac{\xi}{4}q_{n}(p+l)},\\
	\end{split}
	\] 
	with $C$ and $\tilde{C}$ as in \eqref{eq:8.2a} and \eqref{eq:8.2b}, respectively, and $\Delta_{0}(m,p,l,r,s,u)$ defined as in Case 1.
	%
	%
	The only difference is that now we have to build up a factor ${\rm e}^{-\frac{\xi}{4}q_{n}}$
	in the bound of $|\overline{\Val}(\vartheta)|$.
	
\subsubsection*{Case 2.1: $\boldsymbol{p+l \ge 1}$}
	
We proceed as in Case 1.1, up to the fact that ${\rm e}^{-\frac{\xi}{4}q_{n}(p+l)}$ is used to produce the factor ${\rm e}^{-\frac{\xi}{4}q_{n}}$.

\subsubsection*{Case 2.2: $\boldsymbol{p+l = 0}$}

We use the conservation law to write $\nn= \sum_{j=1}^{m} \nn_{\ell_{j}} + \sum_{j=1}^{r} \nn_{\tilde{\ell}_{j}} +\sum_{j=1}^{s} \nn_{\tilde{\ell}'_{j}}$,
so we factorise 
\[
{\rm e}^{\frac{\xi}{2} \big(\sum_{j=1}^{r} |\nn_{\tilde{\ell}_{j}}| +\sum_{j=1}^{s} |\nn_{\tilde{\ell}'_{j}}|\big)}=
{\rm e}^{\frac{\xi}{4} \big(\sum_{j=1}^{r} |\nn_{\tilde{\ell}_{j}}| +\sum_{j=1}^{s} |\nn_{\tilde{\ell}'_{j}}|\big)} {\rm e}^{\frac{\xi}{4}
\big(\sum_{j=1}^{r} |\nn_{\tilde{\ell}_{j}}| +\sum_{j=1}^{s} |\nn_{\tilde{\ell}'_{j}}|\big)}
\]
and use the fact that $|\nn|\ge q_{n}$ in order to obtain
\[
{\rm e}^{-\frac{\xi}{4} \big(\sum_{j=1}^{m} |\nn_{\ell_{j}}|+\sum_{j=1}^{r} |\nn_{\tilde{\ell}_{j}}| +\sum_{j=1}^{s} |\nn_{\tilde{\ell}'_{j}}|\big)}
\le {\rm e}^{-\frac{\xi}{4}|\nn|} \le  {\rm e}^{-\frac{\xi}{4}q_{n}}.
\]
Apart for that, the discussion proceeds as in the previous case 1.2.

\subsubsection*{Case 3: $\boldsymbol{\ell_{0} \in L_{\ge,1}(\vartheta)}$}

Finally, we analyse the case in which the root line is on scale $1$ and  has momentum $\nn$ such that
$|\nn|\ge q_{n}$. We bound
	\[
	\begin{split}
	|\overline{\Val}(\vartheta)|&\le C^{k-1}\frac{4\Gamma}{A} \rho^{-(m+p+l+r+s+u)} \rho^{m+p+l+r+s+u} 
	\eta^{\frac{k-r-s-u-1}{\gotn(\gotn+1)} +\frac{\gotn^{2}-1}{\gotn}(m+p+l)} \eta^{(-\gotn+1)(l+s)} \times \\
	& \qquad \times \eta^{-\gotn+1} \e^{\frac{\gotn}{\gotn+1}r} \,|\zeta|^{u}\, {\rm e}^{-\frac{\xi}{4}q_{\gotn}(p+l)} 
	{\rm e}^{-\frac{\xi}{4} \sum_{j=1}^{m} |\nn_{\ell_{j}}| -\frac{\xi}{2} \big(\sum_{j=1}^{r}|\nn_{\tilde{\ell}_{j}}| + \sum_{j=1}^{s} |\nn_{\tilde{\ell}'_{j}}|\big)}\\
	& \hskip-.5cm \le \rho \, C^{k} \tilde{C}_{1}^{-1} \eta^{\frac{k}{\gotn(\gotn+1)} + \frac{\gotn^{2}-1}{\gotn}} \! \eta^{\Delta_{1}(m,p,l,r,s,u)} \eta^{-\gotn+1}
	{\rm e}^{-\frac{\xi}{4}q_{\gotn}(p+l) -\frac{\xi}{4} \sum_{j=1}^{m} |\nn_{\ell_{j}}| -\frac{\xi}{2} 
	\big(\sum_{j=1}^{r}|\nn_{\tilde{\ell}_{j}}| + \sum_{j=1}^{s} |\nn_{\tilde{\ell}'_{j}}|\big)} \! ,
	\end{split}
	\]
	where $C$ and $\tilde{C}_{1}$ are as in \eqref{eq:8.2a} and \eqref{eq:8.2c}, respectively, and 
	\[
	\begin{split}
	& \Delta_{1}(m,p,l,r,s,u) \\
	& \qquad := \frac{(\gotn^{3} \! + \! \gotn^{2}-\gotn-1)(m+p) \! + \! 
	(\gotn^{2} \! - \! 1)(l \! + \! r) \! + \! u(\gotn^{2} \! + \! \gotn-1) \! - \!  
	s(\gotn^{3}-\gotn \! + \! 1)+\gotn^{3} \! + \! \gotn^{2} \! - \! \gotn}{\gotn(\gotn+1)}.
	\end{split}
	\]
Again we distinguish among three different cases: $p+l=1$, $p+l\ge2$ and $p+l=0$.

\subsubsection*{Case 3.1: $\boldsymbol{p+l=1}$}

Noting that ${\rm e}^{-\frac{\xi}{4}q_{n}(p+l)}= {\rm e}^{-\frac{\xi}{4}q_{n}}$, we have to prove 
\begin{equation}
\label{eq:8.15} 
\tilde{C}_{1}^{-1}\eta^{\Delta_{1}(m,p,l,r,s,u)} {\rm e}^{-\frac{\xi}{4} 
\sum_{j=1}^{m} |\nn_{\ell_{j}}| -\frac{\xi}{2} \big(\sum_{j=1}^{r}|\nn_{\tilde{\ell}_{j}}| + \sum_{j=1}^{s} |\nn_{\tilde{\ell}'_{j}}|\big)} \le 1 .
\end{equation}
%
	
\subsubsection*{Case 3.1.1: $\boldsymbol{p=1, \; l=0}$}
	
	If $s=0$, that is $m+r+u\ge\gotn-1$, one has
	\[
	\Delta_{1}(m,1,0,r,0,u)\ge \frac{\gotn^{2}-\gotn-1}{\gotn+1}>0 ,
	\]
	hence the bound \eqref{eq:8.15} follows. 	
	If $s\ge1$,
	one has
	\[
	\Delta_{1}(m,1,0,r,s,u)\ge \frac{(m+r+u)(\gotn^{2}-1) -s(\gotn^{3}-\gotn+1)-1}{\gotn(\gotn+1)} \ge -\frac{\gotn^{2}+\gotn-1}{\gotn+1}s
	\]
	and the thesis follows once again as
	\[
	\begin{split}
	&\tilde{C}_{1}^{-1}\eta^{\Delta_{1}(m,p,l,r,s,u)} 
	{\rm e}^{-\frac{\xi}{4} \sum_{j=1}^{m} |\nn_{\ell_{j}}| -\frac{\xi}{2} 
	\big(\sum_{j=1}^{r}|\nn_{\tilde{\ell}_{j}}| + \sum_{j=1}^{s} |\nn_{\tilde{\ell}'_{j}}|\big)}
	\\ & \qquad
	\le  \Big(\eta^{-\frac{\gotn^{2}+\gotn-1}{\gotn+1}} \e^{\frac{\xi}{2C_{0}}}\Big)^{s} 
	\le \Big(\eta^{-\frac{\gotn^{2}+\gotn-1}{\gotn+1}+ \frac{\xi}{2C_{0}}}\Big)^{s} <1,
	\end{split}
	\]
	where the last inequality holds if $C_{0}$ satisfies \eqref{eq:8.14}. 

\subsubsection*{Case 3.1.2: $\boldsymbol{p=0, \; l=1}$}
	
	This case is possible only if $|\sum_{j=1}^{m} \nn_{\ell_{j}} + \sum_{j=1}^{r}\nn_{\tilde{\ell}_{j}} |\ge q_{n}$. Indeed, if this were not the case, 
	by using the properties of continued fractions (see Proposition \ref{prop:5.6}), one would have
	\[
	|\oo \cdot \big(\sum_{j=1}^{m} \nn_{\ell_{j}} + \sum_{j=1}^{r} \nn_{\tilde{\ell}_{j}}\big)| \ge \frac{C_{1}}{2}\e^{\frac{1}{\gotn+1}} .
	\]
	and, since the lines $\ell_{0}$ and $\ell''_{1}$ are both on scale $1$, one would obtain the contraddiction
	%
	\[
	\frac{C_{1}}{2}\e^{\frac{1}{\gotn+1}} \le |\oo \cdot \big(\sum_{j=1}^{m} \nn_{\ell_{j}} + 
	\sum_{j=1}^{r} \nn_{\tilde{\ell}_{j}}\big)| = |\oo\cdot (\nn- \nn_{\ell''_{1}})|\le 
	|\oo \cdot \nn|+|\oo\cdot \nn_{\ell''_{1}}|< \frac{C_{1}}{2}\e^{\frac{1}{\gotn+1}}.
	\]
	If $s=0$, one has
	\[
	\Delta_{1}(m,0,1,r,0,u)\ge -\frac{\gotn}{\gotn+1}.
	\]
	so, by choosing $C_0$ according to \eqref{eq:8.14}, one obtains
	\[
	\begin{split}
	\tilde{C}_{1}^{-1}\eta^{\Delta_{1}(m,0,1,r,0,u)} &\, {\rm e}^{-\frac{\xi}{4} 
	\big(\sum_{j=1}^{m} |\nn_{\ell_{j}}|+\sum_{j=1}^{r}|\nn_{\tilde{\ell}_{j}}| \big) -\frac{\xi}{4} \sum_{j=1}^{r}|\nn_{\tilde{\ell}_{j}}| }
	\le \eta^{-\frac{\gotn}{\gotn+1}} {\rm e}^{-\frac{\xi}{4}q_{n}}\\
	& \hskip-1cm \le \eta^{-\frac{\gotn}{\gotn+1}} \e^{\frac{\xi}{4C_{0}}}
	\le \eta^{-\frac{\gotn}{\gotn+1} + \frac{\xi}{4C_{0}}}
	\le \eta^{\gotn-1} \le  1
	\end{split}
	\]
	which yields the bound \eqref{eq:8.15}.	
	If $s\ge1$, one has 
	\[
	\Delta_{1}(m,0,1,r,s,u)\ge -\frac{s(\gotn^{3}+\gotn^{2}-\gotn)+\gotn^{2}}{\gotn(\gotn+1)}
	\ge -\frac{\gotn^{3}+2\gotn^{2} - \gotn}{\gotn(\gotn+1)}\,s = -\frac{\gotn^{2}+2\gotn - 1}{\gotn+1}\,s,
	\]
	so that 
	%
	\[
	\begin{split}
	&\tilde{C}_{1}^{-1}\eta^{\Delta_{1}(m,0,1,r,s,u)} \, e^{-\frac{\xi}{4}\sum_{j=1}^{m} |\nn_{\ell_{j}}| -\frac{\xi}{2} \big(\sum_{j=1}^{r}|\nn_{\tilde{\ell}_{j}}| + \sum_{j=1}^{s} |\nn_{\tilde{\ell}'_{j}}|\big)}
	\\ & \qquad \le  
	\Big(\eta^{-\frac{\gotn^{2}+2\gotn-1}{\gotn+1}} \e^{\frac{\xi}{2C_{0}}}\big)^{s} 
	\le \Big(\eta^{-\frac{\gotn^{2}+2\gotn-1}{\gotn+1} + \frac{\xi}{2C_{0}}} \Big)^{s}
	\le \eta^{s(\gotn-1)}  \le 1 .
	\end{split}
	\] 
	%

\subsubsection*{Case 3.2: $\boldsymbol{p+l \ge 2}$}
		
	We want to prove that
	\begin{equation} 
	\label{eq:8.16} 
	\tilde{C}_{1}^{-1}\eta^{\Delta_{1}(m,p,l,r,s,u)} \times {\rm e}^{-\frac{\xi}{4}q_{n}(p+l)} 
	{\rm e}^{-\frac{\xi}{4} \sum_{j=1}^{m} |\nn_{\ell_{j}}| -\frac{\xi}{2} \big(\sum_{j=1}^{r}|\nn_{\tilde{\ell}_{j}}| + 
	\sum_{j=1}^{s} |\nn_{\tilde{\ell}'_{j}}|\big)} \le {\rm e}^{-\frac{\xi}{4}q_{n}}.
	\end{equation}
	First of all note that
	\[
	{\rm e}^{-\frac{\xi}{4}q_{n}(p+l)}= {\rm e}^{-\frac{\xi}{4}q_{n}} {\rm e}^{-\frac{\xi}{4}q_{n}(p+l-1)},
	\]
	so we can use the first factor to cancel the same factor on the right hand side. 
	By using the fact that $p+l\ge2$, we bound 
	\[
	\Delta_{1}(m,p,l,r,s,u)\ge -\frac{\gotn^{3}-\gotn+1}{\gotn(\gotn+1)}s -\frac{\gotn^{3}-\gotn^{2}-\gotn+2}{\gotn(\gotn+1)} 
	\]
	and we obtain
	\[
	\begin{split}
	 \tilde{C}_{1}^{-1}\eta^{\Delta_{1}(m,p,l,r,s,u)} \, 
	& {\rm e}^{-\frac{\xi}{4} \sum_{j=1}^{m} |\nn_{\ell_{j}}| -\frac{\xi}{2} 
	\big(\sum_{j=1}^{r}|\nn_{\tilde{\ell}_{j}}| + \sum_{j=1}^{s} |\nn_{\tilde{\ell}'_{j}}|\big)} {\rm e}^{-\frac{\xi}{4}q_{n}(p+l-1)}\\
	& \hskip-2cm 
	\le \eta^{-\frac{\gotn^{3}-\gotn^{2}-\gotn+2}{\gotn(\gotn+1)}} \, 
	{\rm e}^{-\frac{\xi}{4}q_{n}} \Big(\eta^{-\frac{\gotn^{3}-\gotn+1}{\gotn(\gotn+1)}} {\rm e}^{-\frac{\xi}{2}q_{n}}\Big)^{s}\\
	& \hskip-2cm 
	\le \eta^{-\frac{\gotn^{3}-\gotn^{2}-\gotn+2}{\gotn(\gotn+1)}} \, \e^{\frac{\xi}{4C_{0}}} 
	\Big(\eta^{-\frac{\gotn^{3}-\gotn+1}{\gotn(\gotn+1)}} \e^{\frac{\xi}{2C_{0}}}\Big)^{s}
	\le \eta^{-\frac{\gotn^{3}-\gotn^{2}-\gotn+2}{\gotn(\gotn+1)} + \frac{\xi}{4C_{0}}} 
	\Big(\eta^{-\frac{\gotn^{3}-\gotn+1}{\gotn(\gotn+1)} + \frac{\xi}{2C_{0}}}\Big)^{s}.\\
	\end{split}
	\]
	So the bound \eqref{eq:8.16} follows if we require $\eta^{-\frac{\gotn^{3}-\gotn^{2}-\gotn+2}{\gotn(\gotn+1)}+ \frac{\xi}{4C_{0}}}<1$, 
	which implies $\eta^{-\frac{\gotn^{3}-\gotn+1}{\gotn(\gotn+1)} + \frac{\xi}{2C_{0}}}<1$ as well.
	In fact, we do not have to add any other condition on $C_{0}$, because if $C_{0}$ satisfies \eqref{eq:8.14}, then 
	$\eta^{-\frac{\gotn^{3}-\gotn^{2}-\gotn+2}{\gotn(\gotn+1)}+ \frac{\xi}{4C_{0}}}< \eta^{\frac{2(\gotn -1)}{\gotn}} \le 1$.
	
\subsubsection*{Case 3.3: $\boldsymbol{p+l =0}$}

First of all we observe that
	%
	${\rm e}^{-\frac{\xi}{4} \big(\sum_{j=1}^{m} |\nn_{\ell_{j}}| + 
	\sum_{j=1}^{r}|\nn_{\tilde{\ell}_{j}}| + \sum_{j=1}^{s} |\nn_{\tilde{\ell}'_{j}}|\big)} \le {\rm e}^{-\frac{\xi}{4}|\nn|} \le {\rm e}^{-\frac{\xi}{4}q_{n}}$,
	%
	as $|\nn|\ge q_{n}$. So we have to prove that
\begin{equation} 
\label{eq:8.17} 
\tilde{C}_{1}^{-1}\eta^{\Delta_{1}(m,0,0,r,s,u)} {\rm e}^{-\frac{\xi}{4} 
\big(\sum_{j=1}^{r}|\nn_{\tilde{\ell}_{j}}| + \sum_{j=1}^{s} |\nn_{\tilde{\ell}'_{j}}|\big)} \le 1.
\end{equation}
If $s=0$, one has $u+m+r\ge\gotn$, that implies
\[
\Delta_{1}(m,0,0,r,0,u) \ge \frac{m(\gotn^{3}-\gotn)+\gotn u -\gotn^{2}}{\gotn(\gotn+1)}.
\]
If $m\ge 1$, $\Delta_{1}(m,0,0,r,0,u)$ is strictly positive. If $m=0$, one has
\[
\Delta_{1}(0,0,0,r,0,u)\ge -\frac{\gotn}{\gotn+1} ;
\]
moreover, in such a case, $\nn= \sum_{j=1}^{r} \nn_{\tilde{\ell}_{j}}$ and $|\nn|\ge q_n\ge C_{0}^{-1}\log \e^{-1}$, so that
\[
\begin{split}
\tilde{C}^{-1}\eta^{\Delta_{1}(0,0,0,r,0,u)}& {\rm e}^{-\frac{\xi}{4} \sum_{j=1}^{r}|\nn_{\tilde{\ell}_{j}}|}
\le  \eta^{-\frac{\gotn}{\gotn+1}} {\rm e}^{-\frac{\xi}{4}|\nn|} 
\le  \eta^{-\frac{\gotn}{\gotn+1}} \e^{\frac{\xi}{4C_{0}}} 
\le \eta^{-\frac{\gotn}{\gotn+1} + \frac{\xi}{4C_{0}}} \le\eta^{\gotn-1} \le 1
\end{split}
\]
and the bound \eqref{eq:8.17} holds once again, provided $C_{0}$ is taken so as to satisfy \eqref{eq:8.14}. 	
If $s\ge1$, one has 
\[ 
\Delta_{1}(m,0,0,r,s,u)\ge -\frac{\gotn^{2}+2\gotn -1}{\gotn+1} s
\]
and  the desired bound follows by  requiring $\eta^{-\frac{\gotn^{2}+2\gotn-1}{\gotn+1} + \frac{\xi}{4C_{0}}} \le1$; indeed, one has
	\[
	\begin{split}
	\tilde{C}_{1}^{-1}&\eta^{\Delta_{1}(m,0,0,r,s,u)} {\rm e}^{-\frac{\xi}{4} 
	\big(\sum_{j=1}^{r}|\nn_{\tilde{\ell}_{j}}| + \sum_{j=1}^{s} |\nn_{\tilde{\ell}'_{j}}|\big)}
	\le \Big(\eta^{-\frac{\gotn^{2}+2\gotn-1}{\gotn+1}} \e^{\frac{\xi}{4C_{0}}}\Big)^{s}
	\le \Big(\eta^{-\frac{\gotn^{2}+2\gotn-1}{\gotn+1} + \frac{\xi}{4C_{0}}}\Big)^{s}\le1.
	\end{split}
	\]
Therefore, to sum up, the third inequality of \eqref{eq:8.5} holds by requiring
\begin{equation}
\label{eq:8.18}
C_{0}=  \frac{\gotn+1}{4(\gotn^{2}+2\gotn-1)}\xi.
\end{equation} 
%

\subsubsection*{Conclusion {\color{black}of the proof}}

By comparing the condition on $C_{0}$ in \eqref{eq:8.14} with \eqref{eq:8.18}, we find
	\begin{equation} \nonumber
	\begin{cases}
	C_{0} \le \displaystyle{ \frac{(\gotn+1)\xi}{4(\gotn^{2}+\gotn-1) } } \\ \vspace{-.4cm} \\
	C_{0} \le \displaystyle{ \frac{(\gotn+1)\xi}{4(\gotn^{2} +2\gotn -1)} }
	\end{cases}
	\quad \Longrightarrow \quad 
	C_{0} \le \frac{(\gotn+1)\xi}{4(\gotn^{2}+2\gotn-1)},  
	\end{equation} 
In conclusion, if we take $C_{0}$ as in \eqref{eq:8.18},
all the bounds in \eqref{eq:8.5} are satisfied and the proof of Lemma \ref{lem:8.2} is completed. 
\end{proof}

The bound in \eqref{eq:8.5} can be simplified as follows:
\begin{equation} 
\label{eq:8.19}
|\overline{\Val}(\vartheta)|\le \rho \, C^{k} \eta^{\frac{k}{\gotn(\gotn+1)} + \frac{\gotn^{2}-1}{\gotn}},
\end{equation}
in such a way that \eqref{eq:8.4} becomes
\begin{equation} 
\label{eq:8.20}
|\Val(\vartheta)| \le  \rho \, C^{k} \eta^{\frac{k}{\gotn(\gotn+1)} + \frac{\gotn^{2}-1}{\gotn}} \prod_{v \in E_{1}(\vartheta)} {\rm e}^{-\frac{\xi}{2}|\nn_{v}|},
\end{equation}
provided $C_{0}$ is as in \eqref{eq:8.18} and $\e \in J_N(C_1,C_0)$, with $N$ satisfying \eqref{eq:2.6b}.
The bound \eqref{eq:8.5} is used in the induction argument, but what we need in the following is the simpler bound \eqref{eq:8.20}.
Note that, by \eqref{eq:8.2a}, one has $C= a_1 C_1$, where the constant $a$ does not depend on $C_1$ if $C_1$ is small enough
(more precisely if $C_1 \le A$). 


\subsection{Properties of the renormalised series}
\label{sec:8.2}

We want to prove that the renormalised series \eqref{eq:6.10} for all $\e\in J_N(C_1,C_0)$ converges,
solves the range equation \eqref{eq:4.2} and is analytic in $\psi$, and it
can be extended to a function which is continuous in $\e$ and goes to $0$ as $\e$ tends to $0$.

\begin{lemma}
	\label{lem:8.3}
	For any $k\ge 1$ and $\nn \in \ZZZ_{*}^{2}$ one has 
	\begin{equation} \nonumber
	|u_{\nn}^{[k]}| \le  \rho B^{k} \eta^{\frac{k}{\gotn(\gotn+1)} + \frac{\gotn^{2}-1}{\gotn}} {\rm e}^{-\frac{\xi}{4}|\nn|},
	\end{equation}
	where $B$ is a positive constant proportional to $C$,
	with $C$ defined as in \eqref{eq:8.2a}, provided $C_{0}$ is as in \eqref{eq:2.4}
	and $\e \in J_{N}(C_{1},C_{0})$, where $N$ satisfies \eqref{eq:2.6b}. 
\end{lemma}

\begin{proof}
	To bound the coefficients $u_{\nn}^{[k]}$ defined as in \eqref{eq:6.9}, we use the estimate  \eqref{eq:8.20} and 
	sum over all trees in $\gotT_{k,\nn}$. 
	The sum over the mode labels $\nn \in \ZZZ^{2}$ in \eqref{eq:6.9} can be performed
	by using the factor ${\rm e}^{-\frac{\xi}{2}|\nn_{v}|}$ associated with end nodes in $E_{1}(\vartheta)$ and this gives a bound 
	$B_{1}^{|E_{1}(\vartheta)|} {\rm e}^{-\frac{\xi}{4}|\nn_{v}|}$ for some positive constant $B_{1}$. 
	The sum over the other labels produces a factor $B_{2}^{k(\vartheta)}$, 
	with $B_{2}$ a suitable positive constant. By taking $B=B_{1}B_{2}C$ the assertion follows.
\end{proof}

\begin{coro}
	\label{coro:8.4}
	Let $C_0$ be as in \eqref{eq:2.4} and let $\e$ belong to $J_{N}(C_{1},C_{0})$, with $N$ satisfying \eqref{eq:2.6}.
	There exists $\zeta_0>0$ such that, if $|\zeta| \le \zeta_0$, then
	the renormalised series \eqref{eq:6.10} which defines $\bar{u}(\pps; \e,\zeta,c)$ converges to a function
	analytic in $\pps$ in a strip $\Sigma_{\xi'}$, with $\xi' < \xi/4$. 
\end{coro}

\begin{proof}
The bound on the Fourier coefficients provided by Lemma \ref{lem:8.3} yields that the renormalised series \eqref{eq:6.10} converges for
$\eta$ small enough, i.e.~such that $|\eta| \le \eta_0$, with $\eta_0=a_0C_1^{\gotn(\gotn+1)}$, for some constant
$a_0$ depending on $\xi,\rho,\Phi,\Gamma$ but not on $C_1$ (see comments after \eqref{eq:8.20}).
Then, taking $\zeta_0=\eta_0$, the result follows.
\end{proof}

\begin{lemma}
	\label{lem:8.5}
	For all $\e\in J_N(C_1,C_0)$, with $C_0$ given by \eqref{eq:2.4} and $N$ satisfying \eqref{eq:2.6},
	there exists $\zeta_0>0$ such that, if $|\zeta| \le \zeta_0$,
	the renormalised series $\bar{u}(\pps; \e, \zeta,c)$ solves the range equation \eqref{eq:8.1}. 
\end{lemma}

\begin{proof}
	Write 
	\begin{equation} \nonumber
	\bar{u}_{\nn} = \sum_{\tilde{n}=0}^{1} \bar{u}_{\nn, \tilde{n}}, \qquad \bar{u}_{\nn, \tilde{n}}= 
	\sum_{k=1}^{\infty} \e^{k} \sum_{\vartheta \in \gotT_{k,\nn,\tilde{n}}} \Val(\vartheta),
	\end{equation}
	where $\gotT_{k, \nn, \tilde{n}}$ is the subset of $\gotT_{k, \nn}$ of the renormalised trees with root line on scale $\tilde{n}$.
	
	Define, for $\tilde{n}=0,1$,
	\begin{eqnarray} \nonumber
		D_{\tilde{n}}(x; \e, \zeta, c) :=  \ii x(1 + \ii \e x) - \mathcal{M}^{[\tilde{n}-1]}(x; \e, \zeta, c), \qquad 
		\mathcal{G}(x; \e, \zeta, c) := \frac{1}{\ii x (1 + \ii x)},
	\end{eqnarray}
	where $\mathcal{M}^{[-1]}(x; \e, \zeta, c)=0$ and $\MM^{[0]}(x;\e,\zeta,c)=\MM(x; \e,\zeta,c)$,
	with $\MM(x; \e,\zeta,c)$ defined in \eqref{eq:6.5a}. 
	Again, in order to simplify the notation we do not make explicit the dependence on $\e$, $\zeta$ and $c$;
	in particular,
	we write $G^{[0]}(x)= \Psi(|x|)/D_{0}(x)$ and $G^{[1]}(x)= \chi(|x|)/D_{1}(x)$, where $G^{[0]}(x)$ and $G^{[1]}(x)$
	are defined in \eqref{eq:6.4}.
	
	If we define 
	\begin{equation}
	\label{eq:8.21} 
	\Omega(\nn;\e,\zeta,c):= \mathcal{G}(\oo \cdot \nn) \e [f - g(c + \zeta + \bar{u}(\cdot; \e,\zeta,c))]_{\nn},
	\end{equation}
	then we have to prove that $\Omega(\nn;\e,\zeta,c)=\bar{u}_{\nn}$ for $\nn\ne0$.
	The right hand side of \eqref{eq:8.21} can be written as
	\begin{equation}
	\label{eq:8.22}
	\begin{split}
	\mathcal{G}(\oo \cdot \nn) &\e [f - g(c + \zeta + \bar{u}(\cdot; \e,\zeta,c))]_{\nn}= \\
	&= \mathcal{G}(\oo \cdot \nn) \bigl( \Psi(|\oo \cdot \nn|) + \chi(|\oo \cdot \nn|) \bigr)
	\e [f - g(c + \zeta + \bar{u}(\cdot;\e,\zeta,c))]_{\nn}\\
	&= \calG(\oo \cdot \nn) \, D_{0}(\oo \cdot \nn) G^{[0]}(\oo \cdot \nn)
	\, \e\, [f - g(c + \zeta + \bar{u}(\cdot,\e,\zeta,c))]_{\nn} + \\
	& \qquad +  \calG(\oo \cdot \nn) \,	D_{1}(\oo \cdot \nn) G^{[1]}(\oo \cdot \nn)\big)\, \e\, [f - g(c + \zeta + \bar{u}(\cdot,\e,\zeta,c))]_{\nn} ,
	\end{split}
	\end{equation}
	where
	\begin{subequations} \label{eq:8.23}
	\begin{align}
	& G^{[0]}(\oo \cdot \nn)\, \e\, [f - g(c + \zeta + \bar{u}(\cdot,\e,\zeta,c))]_{\nn}= 
	\sum_{k=1}^{\infty} \e^{k} \sum_{\vartheta \in \gotT_{k, \nn, 0}} \Val(\vartheta),
	\label{eq:8.23a} \\
	& G^{[1]}(\oo \cdot \nn)\, \e\, [f - g(c + \zeta + \bar{u}(\cdot,\e,\zeta,c))]_{\nn}
	= \sum_{k=1}^{\infty} \e^{k} \sum_{\vartheta \in \gotT^{*}_{k, \nn, 1}} \Val(\vartheta),
	\label{eq:8.23b}
	\end{align}
	\end{subequations}
	where $\gotT^{*}_{k, \nn, 1}$ differs from $\gotT_{k, \nn, 1}$ as it contains also trees which can
	have one renormalised self-energy cluster on scale $0$ with exiting line $\ell_{0}$, if $\ell_{0}$ denotes the root line of $\vartheta$. 
	Indeed, if we analyse the contribution $G^{[1]}(\oo\cdot\nn) \, \e\, [f - g(c + \zeta + \bar{u}(\cdot,\e,\zeta,c))]_{\nn}$,
	we see that it differs from $\bar{u}_{\nn,1}$ in as much as it contains an additional contribution:
	%
	\begin{equation} \label{eq:8.24}
	\sum_{k=1}^{\infty} \e^{k} \sum_{\vartheta \in \gotT^{*}_{k, \nn, 1}} \Val(\vartheta) =
	\sum_{k=1}^{\infty} \e^{k} \sum_{\vartheta \in \gotT_{k, \nn, 1}} \Val(\vartheta) +   
	G^{[1]}(\oo \cdot \nn) M(\oo \cdot \nn) \sum_{k=1}^{\infty} \e^{k} \sum_{\vartheta \in \gotT_{k, \nn, 1}} \Val(\vartheta).
	\end{equation}
	By inserting \eqref{eq:8.23} and \eqref{eq:8.24} into \eqref{eq:8.22}, we obtain
	\[
	\begin{split}
	\Omega(\nn&;\e,\zeta,c)= \calG(\oo \cdot \nn) \e [f - g(c + \zeta + \bar{u}(\cdot,\e,\zeta,c))]_{\nn}  \\
	&= \calG(\oo \cdot \nn) \Big(D_{0}(\oo \cdot \nn)\sum_{k=1}^{\infty} \e^{k} \sum_{\vartheta \in \gotT_{k, \nn, 0}} \Val(\vartheta)  
	+ D_{1}(\oo \cdot \nn) \sum_{k=1}^{\infty} \e^{k} \sum_{\vartheta \in \gotT_{k, \nn, 1}} \Val(\vartheta)+\\ 
	& \qquad+ D_{1}(\oo \cdot \nn) G^{[1]}(\oo \cdot \nn) M(\oo \cdot \nn) \sum_{k=1}^{\infty} \e^{k} \sum_{\vartheta \in \gotT_{k, \nn, 1}} \Val(\vartheta)\Big) \\
	&= \calG(\oo \cdot \nn) \Big(D_{0}(\oo \cdot \nn) \bar{u}_{\nn, 0} + D_{1}(\oo \cdot \nn) \bar{u}_{\nn,1}
	\,+ D_{1}(\oo \cdot \nn) G^{[1]}(\oo \cdot \nn) M(\oo \cdot \nn) \bar{u}_{\nn, 1}\Big) \\
	&= D_{0}^{-1}(\oo \cdot \nn) \Big(D_{0}(\oo \cdot \nn) \bar{u}_{\nn, 0} + (D_{1}(\oo \cdot \nn)+\,
	\chi(|\oo\cdot\nn|) M(\oo \cdot \nn)) \bar{u}_{\nn, 1}\Big) \\
	&= \bar{u}_{\nn,0} +  \bar{u}_{\nn,1} =  \bar{u}_{\nn},\\
	\end{split}
	\]
	which completes the proof.
\end{proof}


The following result implies that the renormalised series $\bar{u}(\pps; \e,\zeta,c)$ is continuous in $\e$ in the sense of Whitney for $\e$ close to $0$
{\color{black}
(see Section \ref{sec:2} for the notion of continuity in the sense of Whitney).}

\begin{lemma}
	\label{lem:8.6}
	There exist two positive constants $\e_0$ and $\zeta_0$ such that, if $|\zeta| \le \zeta_0$, 
	the function $\bar{u}(\pps; \e,\zeta,c)$ can be extended to a function $\tilde{u}(\pps; \e,\zeta,c)$, defined for all $\e \in [0,\e_{0}]$, such that 
	\vspace{-.2cm}
	\begin{enumerate}
	\itemsep0em
	\item $\tilde{u}(\pps; \e,\zeta,c) = \bar{u}(\pps; \e,\zeta,c)$ for $\e \in J_{N}(C_{1},C_{0})$,
	\item $\tilde{u}(\pps; \e,\zeta,c)$ is continuous in $\e$,
	\item $\tilde{u}(\pps; \e,\zeta,c) \to 0$ as $\e \to 0^+$.
	\end{enumerate}
\end{lemma}

\begin{proof}
	Continuity of the function $\e \mapsto \bar{u}(\pps; \e,\zeta,c)$ holds trivially for $\e>0$ which belongs to the interior of
	the set $J_{N}(C_{1},C_{0})$ as defined in \eqref{eq:2.3}. As such a set may contain infinitely many holes accumulating to $0$,
	the limit as $\e \to0^+$ requires some discussions. 
	Define
	\begin{equation} \nonumber 
	\calF(\e, \zeta):= || \bar{u}(\cdot; \e,\zeta,c)||_{\infty}= \sup \{\bar{u}(\pps; \e,\zeta,c): \pps \in \Sigma_{\xi'}\}, 
	\end{equation}
	with $\xi'$ as in Corollary \ref{coro:8.4}. 
	Since $\calF(0,\zeta)=0$ by construction -- see \eqref{eq:4.2} and take $\e=0$ -- we need prove that
	$\calF(\e,\zeta) \to 0$ as $\e \to 0$ along any sequence $\{\e_k\}$ contained in $J_{N}(C_{1},C_{0})$ such that $\e_k\to 0^+$,
	i.e.~that for all $\iota>0$ there exists $\delta>0$ such that $\e\in\{\e_k\}$ and $0<\e<\delta$ imply $|\calF(\e,\zeta)|<\iota$.
		We have 
	\begin{equation} \nonumber
	\calF(\e,\zeta) \le \sum_{k=1}^{\io} \sum_{\nn\in\ZZZ^{d}} \bigl| u^{[k]}_{\nn} \bigr| \,  {\rm e}^{\xi'|\nn|} ,
	\end{equation}
	where $u^{[k]}_{\nn}$ is given by the sum of the values of the renormalised trees \eqref{eq:6.9}.
	We can bound the value of a tree with $k$ nodes by using \eqref{eq:8.20} and noting that $E_{1}(\vartheta)\neq \emptyset$
(	otherwise $\nn$ would vanish). For the nodes $v \in E_{1}(\vartheta)$ we use the bound in \eqref{eq:8.3} without
	estimating $\e$ with $\eta$. Since there is at least one such node, we obtain
	\begin{equation}\nonumber
	\begin{split} 
	\sum_{\nn\in\ZZZ^{d}} \bigl| u^{[k]}_{\nn} \bigr| \,  {\rm e}^{\x'|\nn|}
	&  \le 
	\sum_{\nn\in\ZZZ^{d}} \sum_{\vartheta \in \gotT_{\nn,k}} \,\,\bigl| \Val(\vartheta) \bigr| \, {\rm e}^{\x'|\nn|}\\
	&\le
	\rho \, C^{k} \eta^{\frac{k}{\gotn(\gotn+1)} + \frac{\gotn^{2}-1}{\gotn}} \Big(\frac{\e}{\eta}\Big)^{\frac{\gotn}{\gotn+1}}
	\sum_{\vartheta \in \gotT_{\nn,k}} \, 
	\prod_{v \in E_{1}(\vartheta)}
	\sum_{\nn_v \in \ZZZ^{d}}  {\rm e}^{-\xi |\nn_v|/4} \\
	& \le 
	\rho B^{k} \eta^{\frac{k}{\gotn(\gotn+1)} + \frac{\gotn^{3}-\gotn -1}{\gotn(\gotn+1)}}  \e^{\frac{\gotn}{\gotn+1}}
	\sum_{\nn \in \ZZZ^{d}}  {\rm e}^{-\xi |\nn|/4} .
	\end{split}
	\end{equation}
	where the propagator of the line exiting any end node $w \in E_{1}(\vartheta)$ has been bounded either
	with $\e^{\frac{\gotn}{\gotn+1}} {\rm e}^{-\xi|\nn_{w}|}$ or with $\eta^{-\gotn+1} {\rm e}^{-\xi |\nn_{w}|}$,
	by \eqref{eq:8.3}, so that
	we can extract a factor ${\rm e}^{-\frac{\xi}{2}|\nn_{w}|}$, so as to obtain
	\[
	{\rm e}^{-\frac{\xi}{2}|\nn_{w}|}
	{\rm e}^{\xi'|\nn|} \prod_{v \in E_{1}(\vartheta)\setminus \{w\}} {\rm e}^{-\frac{\xi}{2}|\nn_{v}|} \le
	\prod_{v \in E_{1}(\vartheta)} {\rm e}^{-\frac{\xi}{4}|\nn_{v}|} ,
	\]
	and use that $\eta \ge \e$ and $C_0$ is defined as in \eqref{eq:8.18} to bound
	\[
	\eta^{-\gotn+1} {\rm e}^{-\frac{\xi}{2} |\nn_{w}|} \le \eta^{-\gotn+1} {\rm e}^{-\frac{\xi}{2} q_{n}} \le
	\eta^{-\gotn+1} \e^{\frac{2(\gotn^{2}+2 \gotn - 1)}{\gotn+1}}
	\le \e^{\frac{\gotn^{2} + 4\gotn -1}{\gotn+1}} \le \e^{\frac{\gotn}{\gotn+1}} .
	\]
	Hence we obtain the following bound for $\calF(\e,\zeta)$:
	\begin{equation} \nonumber
	\calF(\e,\zeta) \le \rho \, \Bigl(1 - B \eta^{\frac{1}{\gotn(\gotn +1)}} \Bigr)^{-1} \e^{\frac{\gotn}{\gotn+1}}
	\eta^{\frac{\gotn^{3}-\gotn -1}{\gotn(\gotn+1)}} 
	\sum_{\nn \in \ZZZ^{d}}  {\rm e}^{-\xi |\nn|/4} 
	\end{equation}
	and, for fixed $\iota>0$, by choosing
	$\delta>0$ suitably small  and  taking $\e=\e_k$ with $0<\e_k < \delta$ we have $\calF(\e,\zeta) \le \iota$. 
	
	{\color{black}
	By reasoning ia a similar way, one proves that, for all $\e,\e' \in J_{N}(C_{1},C_{0})$,
	the function $\calF$ satisfies the bound
	$|\calF(\e,\zeta)-\calF(\e',\zeta)| \le \om(\e,\e')$ for a suitable modulus of continuity $\om$;
	the bounds above ensures that $\e \mapsto \calF(\e,\zeta)$ is at least H\"older-continuous
	with exponent $\gotn/(\gotn+1)$.}
	
	Therefore, in the light of the bounds satisfied by $\calF(\e,\zeta)$ for $\e\in J_N(C_{1},C_{0})$,
	the function $\bar{u}(\pps; \e,\zeta,c)$ can be extended in the sense of Whitney
	to a function $\tilde{u}(\pps; \e,\zeta,c)$, defined for all $\e \in [0,\e_{0})$, with
	$\e_0$ equal to right endpoint of $J_{N}(C_{1},C_{0})$, and such that  $\tilde{u}(\pps; \e,\zeta,c) = \bar{u}(\pps; \e,\zeta,c)$
	for $\e \in J_{N}(C_{1},C_{0})$.	Therefore $\tilde{u}(\pps; \e,\zeta,c)$ represents  the Whitney extension of $\bar{u}(\pps; \e,\zeta,c)$
	to the interval $[0,\e_0]$ and is continuous in $\e$ by construction, in particular $\tilde{u}(\pps; \e,\zeta,c) \to 0$ as $\e \to 0$. 
\end{proof}

\zerarcounters
\section{The bifurcation equation}
\label{sec:9}

Define 
\begin{equation}
\label{eq:9.1}
\calH(\e, \zeta):= [g(c + \zeta + \bar u(\cdot; \e, \zeta, c))]_{\vzero} - f_{\vzero} ,
\end{equation} 
so that the bifurcation equation in \eqref{eq:4.3} becomes $\calH(\e, \zeta)=0$.
By Hypothesis \ref{hyp:1}, one has 
\begin{equation}
\label{eq:9.2}
\calH(\e,\zeta)= \sum_{p=\gotn}^{\io} g_{p}(c)[(\zeta + \bar u(\cdot; \e, \zeta, c))^{p}]_{\vzero} =
\sum_{k=\gotn +1}^{\io}  \sum_{\vartheta \in \gotT_{k, \vzero}} \Val^{*}(\vartheta) ,
\end{equation} 
where $\Val^{*}(\vartheta)$ is defined as $\Val(\vartheta)$ with the only difference that the node factor of the first node $v_0$
is $F_{v_0}= g_{p_{v_0}}$ (without the factor $-\e$ appearing in \eqref{eq:6.3}).

\begin{lemma}
\label{lem:9.1} 
For any tree $\vartheta \in \gotT_{k, \vzero}$ such that at least one line entering the first node either has scale 1
or does not exit an end node, one has 
\begin{equation} \nonumber
|\Val(\vartheta)| \le \rho \, \Gamma \, C^{k} \eta^{\alpha k + 
\gotn}, \qquad 
\alpha:= \frac{1}{\gotn (\gotn+1)} .
\end{equation}
with $C$ as in \eqref{eq:8.2a} and $\rho, \Gamma$ as in \eqref{eq:2.2}.
\end{lemma}

\begin{proof}
The subtrees which enter the first node $v_0$ of $\vartheta$ are arranged as described at the beginning of the proof of Lemma \ref{lem:8.2}.
Their values are bounded according to Lemma \ref{lem:8.1} and Lemma \ref{lem:8.2}, so that, using
that the root line has momentum $\nn=0$ and the corresponding propagator is $G_{\ell}=1$, one can bound
\[
\begin{split}
|\Val^*(\vartheta)| &\le \Gamma \rho^{-(m+p+l+r+s)}\rho^{m+p+l+r+s}C^{k-1} 
\eta^{\frac{k-r-s-u-1}{\gotn(\gotn+1)} + \frac{\gotn^{2}-1}{\gotn}(m+p+l)} \e^{\frac{\gotn}{\gotn+1}r} \times \\
& \qquad \times \eta^{(-\gotn+1)(l+s)} |\zeta|^{u}  {\rm e}^{-\frac{\xi}{4} 
\sum_{j=1}^{m} |\nn_{\ell_{j}}|} {\rm e}^{-\frac{\xi}{4}q_{n}(p+l)} {\rm e}^{-\frac{\xi}{2} 
\big(\sum_{j=1}^{r} |\nn_{\tilde{\ell}_{j}}| +\sum_{j=1}^{s} |\nn_{\tilde{\ell}'_{j}}|\big)}\\
&\le \rho \Gamma C^{k} \bar{C}^{-1}
\eta^{\frac{k}{\gotn(\gotn+1)} +\Delta_{2}(m,p,l,r,s,u) + \gotn} 
{\rm e}^{-\frac{\xi}{4}q_{n}(p+l+s)} {\rm e}^{-\frac{\xi}{4} q_n s } ,
\end{split}
\]
with $C$ as in \eqref{eq:8.2a}, $\bar{C}:= \rho \, C \ge 1$ and
\[
\begin{split}
\Delta_{2}(m,p,l,r,s,u) 
:= \frac{(m \!+\! p \!+\! l \!+\! r \!+\! u)(\gotn^{2}-1) \!+\! (m+p)(\gotn^{3}-\gotn) \!+\! \gotn u
\!-\! s(\gotn^{3} \!-\! \gotn \!+\! 1) \!-\! \gotn^3 \!-\! \gotn^2 \!-\! 1 }{\gotn(\gotn+1)} . 
\end{split}
\]
Since $m+p+l+r+u+s \ge \gotn$, one has
\[
\Delta_{2}(m,p,l,r,s,u) \ge 
- \frac{(\gotn^{2} +  \gotn + 1 ) + m (\gotn^3 - \gotn ) + u \, \gotn}{\gotn(\gotn+1)} -  \frac{s(\gotn^{2}+\gotn-1)}{\gotn+1} ,
\]
and, taking $C_{0}$ as in \eqref{eq:8.18} and $\eta$ small enough, one bounds
\[
\eta^{- \frac{\gotn^{2}+\gotn-1}{\gotn+1}} {\rm e}^{-\frac{\xi}{4} q_n } 
\le \eta^{- \frac{\gotn^{2} + \gotn -  1}{\gotn+1} + \frac{\xi}{4C_{0}}} \le1, \qquad
\eta^{- \frac{\gotn^{2}+\gotn+1}{\gotn(\gotn+1)}} {\rm e}^{-\frac{\xi}{4} q_n } 
\le \eta^{- \frac{\gotn^{2} + \gotn + 1}{\gotn(\gotn+1)} + \frac{\xi}{4C_{0}}} \le1,
\]
which implies the assertion as long as $p+l+s\ge 1$. If $p+l+s=0$ and $m=1$, one estimates $m(\gotn^3-\gotn) -
(\gotn^2 + \gotn +1) \ge \gotn^3 - \gotn^2 - 2 \gotn - 1 \ge 0$ for $\gotn \ge 3$, so that the result follows once more. Finally,
if $m+p+l+s=0$, one has $p_{v_0}=u+r$, so that all the lines entering $v_0$ have scale 0 and exit an end node.
\end{proof}

\begin{lemma}
\label{lem:9.2}
The function $\calH(\e,\zeta)$ is $C^{\gotn}$ with respect to $\zeta$. 
\end{lemma}

\begin{proof}
By \eqref{eq:9.2}, it is sufficient to prove that $\Val^*(\vartheta)$ is $C^{\gotn}$ in $\zeta$ for all $\gotT_{k, \vzero}$.
Note that $\Val^*(\vartheta)$ depends on $\zeta$ through the node factors and through the propagators associated
with the lines on scale $1$ -- see \eqref{eq:6.3} and \eqref{eq:6.4b}.
For $j=1,\ldots,\gotn$, one has
\begin{equation} \label{eq:9.3} 
\partial^j_{\zeta} \Val^*(\vartheta) = \sum_{j_1+j_2=j}
\Big[ \Bigl( \partial^{j_1}_{\zeta} \Big(\prod_{v \in N(\vartheta)} F_{v}\Big) \Bigr)
\partial^{j_2} _{\zeta}\Big(\prod_{\ell \in L(\vartheta)} G_{\ell}\Big)\Big] .
\end{equation}
%
%
%

The derivatives acting on the node factors are easily controlled:
\begin{equation} \nonumber
\Big|\partial^{j_1}_{\zeta} \Big(\prod_{v \in N(\vartheta)} F_{v} \Big) \Big| 
\le |E_{0}(\vartheta)| (|E_{0}(\vartheta)|-1) \cdots (|E_{0}(\vartheta)|-j_1 +1)
|\zeta|^{ |E_{0}(\vartheta)| - j_1} \!\!\!\!\!\! \prod_{v \in N(\vartheta) \setminus E_{0}(\vartheta)} \bigl| F_{v} \bigr| .
\end{equation}
%
%
%
	
Dealing with the derivatives acting on the propagators is more delicate.
Consider first the trees such that at least one line entering the first node either has scale 1 or does not exit an end node.
Each propagator $G^{[1]}(\oo\cdot\nn)$ can be differentiated at most $j_2$ times.
For $p \le j_2 \le \gotn$, if one sets $x=\oo\cdot\nn$ and
\begin{equation} \nonumber
\mathcal{A}_{p,k} :=\{i_{1},\dots, i_{k} \in \NNN :  
i_{1}\ge i_{2}\ge \dots\ge  i_{k}\ge 1 \hbox{ and } i_{1}+\dots+ i_{k}=p \} ,
\end{equation}
one has
\begin{equation} \label{eq:9.4} 
	%
\partial^{p}_{\zeta} G^{[1]}(x) = \chi(|x|) \sum_{k=1}^{p} \sum_{i_{1},\dots, i_{k} \in \mathcal{A}_{p,k} } 
a_{i_{1},\dots,i_{k}} \frac{ (\partial^{i_{1}}_{\zeta} \MM(x) ) ( \partial^{i_{2}}_{\zeta} \MM(x)) \ldots
( \partial^{i_{k}}_{\zeta} \MM(x) ) }{(D_{1}(x))^{k+1}}, 
\end{equation}
for suitable combinatorial coefficients $a_{i_{1},\dots,i_{k}}$.  Note that in \eqref{eq:9.4}, one has $\chi(|x|)/D_1(x)=G^{[1]}(x)$.
%
Since the lines of the self-energies contributing to $\MM(x)$ are on scale $0$ (see \eqref{eq:6.5a}), in each factor
$\partial^{i_{s}}_{\zeta} \MM(x)$, with $s=1,\ldots,p-k$, appearing in \eqref{eq:9.4}
the derivatives with respect to $\zeta$ act only on the node factors.  
By using the bounds of  Lemma \ref{lem:7.1} and Lemma \ref{lem:7.2}, one finds
\[
|\partial^{j}_{\zeta}\MM(x)| \le b_{j} \e \Bigl( |\zeta|^{\gotn- j - 1}  + \eta^{\frac{2\gotn^{2}-j(\gotn+1)-\gotn}{\gotn+1}} \Bigr) ,
\quad j=1,\ldots,\gotn-1,
\qquad
|\partial^{\gotn}_{\zeta}\MM(x)| \le b_{\gotn} \e \Bigl( 1+ \eta^{\frac{\gotn^{2}-2\gotn}{\gotn+1}} \Bigr)  ,
\]
for suitable positive constants $b_{1},\dots,b_{\gotn}$; in each bound, the first contribution arises from resonances
that contain only one internal node, while the second one takes into account the resonances with at least two internal nodes.
Hence, to sum up, 
one has
\begin{equation} \nonumber
%
|\partial^{j}_{\zeta}\MM(x)| \le 2b_{0} \e \eta^{\gotn-j-1}, \quad j=1,\dots,\gotn,\qquad b_0:=\max\{ b_1, \ldots, b_{\gotn} \} ,
\end{equation}
which can be used to bound the derivatives in \eqref{eq:9.4} so as to give
\begin{eqnarray} \nonumber 
& & \Biggl|\frac{(\partial^{i_{1}}_{\zeta}\MM(x))(\partial^{i_{2}}_{\zeta}\MM(x))\dots(\partial^{i_{k}}_{\zeta}\MM(x))}{(D_{1}(x))^{k}}\Biggr| 
\le \frac{(2b_0\e)^k \eta^{(\gotn-1)k-(i_1+\ldots+i_k)}}{(A\e \eta^{\gotn-1})^k} \le
\Bigl( \frac{2b_0}{A} \Bigr)^ k \eta^{-p} ,
\end{eqnarray}
for all values of $k$ and $p$.
%
%
Coming back to \eqref{eq:9.3}, one can write
\[
\partial^{j_2} _{\zeta}\Big(\prod_{\ell \in L(\vartheta)} G_{\ell}\Big)  = \sum_{r=1}^{j_2}
\sum_{p_1+\ldots+p_r = j_2} 
\partial_{\zeta}^{p_{1}}G^{[1]}(\oo\cdot\nn_{\ell_{1}}) \ldots 
\partial_{\zeta}^{p_{r}}G^{[1]}(\oo\cdot\nn_{\ell_{r}})
\Big( \prod_{\ell \in L(\vartheta)\setminus \{\ell_1,\ldots,\ell_r\}} G_{\ell} \Big),
\]
where each $\partial_{\zeta}^{p_{i}} G^{[1]}(\oo\cdot\nn_{\ell_{i}})$, $i=1,\ldots,r$, can be dealt with according
to \eqref{eq:9.4}: it is given by a sum of terms, each of which contains $r$ factors which admit the same bound as $G^{[1]}(\oo\cdot\nn)$
times a factor which can be bounded proportionally to $\eta^{-p_{i}}$.
The number of summands arising from the sums over the choice of the lines 
which are differentiated (sum over $r$ and $p_1,\ldots,p_r$) and from the sums in \eqref{eq:9.4} is bounded by $K_2^k$,
for a suitable constant $K_2$. In conclusion, one has
\begin{equation} \nonumber 
\Bigl| \partial^{j_2} _{\zeta} \Big(\prod_{\ell \in L(\vartheta)} G_{\ell}\Big) \Bigr| \le K_2^k \eta^{ -j_2}
\Bigl| \prod_{\ell \in L(\vartheta)} G_{\ell} \Bigr| .
\end{equation}
which, together with the bound \eqref{eq:9.3} and Lemma \ref{lem:9.1}, implies, for another suitable constant $K_1$,
\begin{equation} \nonumber
\Bigl( \partial^{j_1}_{\zeta} \Big(\prod_{v \in N(\vartheta)} F_{v}\Big) \Bigr)
\partial^{j_2} _{\zeta}\Big(\prod_{\ell \in L(\vartheta)} G_{\ell}\Big) \le
\eta^{-j} K_1^{k} K_2^k
\max \{ |\Val^*(\vartheta)| : \vartheta \in \TT_{k,\vzero} \} \le
\rho \, \Gamma \, (C K_1 K_2)^{k} \eta^{\alpha k + \beta} ,
\end{equation}
which shows that the first $\gotn$ derivatives are bounded, provided $\eta$ is small enough.

We are left with the trees $\theta\in\TT_{k,\vzero}$ for which, by using the notation introduced at the beginning of the proof
of Lemma \ref{lem:8.2}, one has $m+p+l+s=0$, so that $u+r \ge \gotn$: all the subtrees entering the first node $v_0$
of $\vartheta$ have order 1 and root line on scale 0. The only dependence on $\zeta$ is through the factors $\zeta$ associated
to the end nodes in $E_0(\vartheta)$. This means that the derivative $\partial_\zeta^j \Val^*(\vartheta)$ is not zero if and only if
$u \ge j$ and in such a case one finds
\begin{equation} \nonumber
\left| \partial_{\zeta}^j  \Val^*(\vartheta) \right|  \le \rho K_1^j \Gamma C^{u+r} |\zeta|^{u-j} \e^{\frac{\gotn r}{\gotn+1}} ,
\end{equation}
where the bounds of Lemma \ref{lem:8.1} have been used.
Thus the first $\gotn$ derivatives are bounded in such a case as well.
\end{proof}

\begin{lemma}
	\label{lem:9.3}
	Let $J_N(C_1,C_0)$ as in Section \ref{sec:8}.
	There exists a neighbourhood $\calU \times \calV$ of $(\e, \zeta)=(0,0)$ such that for all $\e \in \calU \cap J_N(C_1,C_0)$
	there is at least one value $\zeta=\zeta(\e) \in \calV$,
	depending continuously on $\e$, which solves the bifurcation equation \eqref{eq:4.3}.
\end{lemma}

\begin{proof}
We consider explicitly the case $g_{\gotn} > 0$ only, since the case $g_{\gotn}< 0$ is discussed in the same way.
By Hypothesis \ref{hyp:1}, the function $\calH(0,\zeta)$ has in $\zeta=0$ a rising point of inflection.
In particular, since $\calH(0,0)=0$, there exists an interval $\calV=[V_-,V_+]$ such that  $\calH(0,V_{+})> 0$ and $\calH(0, V_{-})<0$. 
	
Now we look at $\calH(\e,V_{-})$ and $\calH(\e,V_{+})$ as  functions of $\e$: they are both continuous (in the sense of Whitney)
in a neighbourhood of $\e=0$, so that there exists an interval $\calU$ such that  
$\calH(\e,V_{-})<0$ and $\calH(\e,V_{+})>0$ for all $\e \in \calU$ for which the functions are defined.
	
Since for all $\e \in \calU$ the function $\calH(\e,\zeta)$ is defined for $\zeta\in \calV$
and one has $\calH(\e,V_{-})<0$ and $\calH(\e,V_{+})>0$, then, by continuity, there exists a curve 
$\e \in J_N(C_1,C_0) \mapsto \zeta(\e) \in \calV$ such that $\calH(\e,\zeta(\e))=0$. 
\end{proof}

%

\zerarcounters 
\section{Conclusions}
\label{sec:10}

Theorem \ref{thm:1} is a bit unsatisfactory, since an arbitrary constant $C_1$ is involved,
in terms of which the maximal value allowed for $\e$ is expressed. Moreover the sets $I_n(C_1,C_0)$
are all fixed once and for all when the value of $C_1$ is given, whereas one can imagine that one can reduce
the sizes of the holes by taking a different value of $C_1$ for each interval $I_n(C_1,C_0)$.

In fact, one can improve the construction envisaged in Section \ref{sec:2} as follows.
First of all note that $C_0$ is defined uniquely according to \eqref{eq:2.4}.
Then one fix the constant $C_1$ and the first interval $I_N(C_1,C_0)$ as follows.
Define $I_N(C',C)$ as in \eqref{eq:2.3}: for $\e\in I_N(C_1,C_0)$ the renormalised series converges
provided the condition $|\e|<\e_0:=\eta_0 C_1^{\gotn(\gotn-1)}$ is satisfied, which requires
\begin{equation} \nonumber 
\frac{1}{(C_1 q_N)^{\gotn + 1}} \le \eta_0 C_1^{\gotn(\gotn-1)} .
\end{equation}
This suggests taking $C_1$ as the value for which the equal sign holds, so that one finds
\begin{equation} \label{eq:10.1}
C_1 = C_1^* := \Bigl( \frac{1}{a_0 q_N} \Bigr)^{\frac{1}{\gotn + 1 }} , \qquad a_0 := \eta_0^{\frac{1}{\gotn + 1 }} ,
\end{equation}
and the first interval becomes
\begin{equation} \nonumber 
\gotI_{0} := I_{N}(C_1^*,C_0) = \Bigl[ {\rm e}^{-C_0 q_N } , \frac{a_0}{q_N^{\gotn}} \Bigl] .
\end{equation}
The value of $N$ is still undetermined: we choose $N$ by requiring that
\begin{equation} \label{eq:10.2}
{\rm e}^{-C_0 x } \le \frac{a_0}{x^{\gotn}} \quad \forall x \ge q_N .
\end{equation}
This means that $q_N$ is chosen as the lowest denominator of the convergents for which \eqref{eq:10.2} is satisfied.

Next, we consider the interval $I_{N+1}(C_2,C_0)$: for $\e\in I_{N+1}(C_2,C_0)$ the renormalised series converges provided
one has $|\e|<\e_1:=\eta_0 C_2^{\gotn(\gotn-1)}$. Again, in order to maximise the size of the interval, we fix $C_2$ by requiring
\begin{equation} \nonumber 
\frac{1}{(C_2 q_{N+1})^{\gotn + 1}} \le \eta_0 C_2^{\gotn(\gotn-1)} .
\quad \Longrightarrow \quad
C_2 = C_2^* := \Bigl( \frac{1}{a_0 q_{N+1}} \Bigr)^{\frac{1}{\gotn + 1 }} ,
\end{equation}
and set
\begin{equation} \nonumber 
\gotI_{1} := I_{N+1}(C_2^*,C_0) = \Bigl[ {\rm e}^{-C_0 q_{N+1} } , \frac{a_0}{q_{N+1}^{\gotn}} \Bigl] .
\end{equation}
Note that $C_2^*> C_1^*$, so that $\gotI_{1} \supset I_{N+1}(C_1,C_0)$.

We iterate the construction above by defining
\begin{equation} \label{eq:10.3}
\gotI_{n} := \Bigl[ {\rm e}^{-C_0 q_{N+n} } , \frac{a_0}{q_{N+n}^{\gotn}} \Bigl] , \qquad
\gotJ = 
{\color{black}
\mbox{Cl} 
}\Bigl( \bigcup_{n=0}^{+\io} \gotI_n \Bigr) .
\end{equation}
For all $\e\in \gotJ$ the renormalised series converges and, by reasoning as in the 
previous sections, one proves that one can fix $\zeta$ as a function of $\e$ in such a way that
both the bifurcation and the range equations are satisfied.

We summarise the discussion above in the following statement.

\begin{thm} \label{thm:2}
Consider the ordinary differential equation \eqref{eq:1.1}, with $f$ analytic in the strip $\Sigma_{\xi}$ and $\oo=(1, \alpha)$,
and assume Hypotheses \ref{hyp:1} and \ref{hyp:2}. Denote by $p_{n}/q_{n}$ the convergents of $\alpha$,
and let $C_{0}$ be as in \eqref{eq:2.4} and $q_N$ such that \eqref{eq:10.2} is satisfied, with $a_0$ as in \eqref{eq:10.1}.
Then for all $\e \in \gotJ$ there is at least one quasi-periodic solution $x(t,\e)= c + X(\oo t, \e)$ to \eqref{eq:1.1}, such that
$X(\pps, \e)$ is analytic in $\pps$ in the strip $\Sigma_{\xi'}$, with $\xi'<\xi/4$, is continuous in $\e$ in the sense of Whitney,
and goes to zero as $\e \to 0$. 
\end{thm} 

The intervals \eqref{eq:10.3} overlap without leaving holes if and only if one has
\begin{equation} \label{eq:10.4}
q_{n+1} \le b_0 {\rm e}^{(C_0/\gotn) q_n } \qquad \forall n \ge N , \qquad b_0:= \eta_0^{\frac{1}{\gotn(\gotn+1)}} .
\end{equation}
This leads to the following result.

\begin{thm} \label{thm:3}
Consider the ordinary differential equation \eqref{eq:1.1}, with $f$ analytic in the strip $\Sigma_{\xi}$ and $\oo=(1, \alpha)$,
and assume Hypothesis \ref{hyp:1}. Denote by $p_{n}/q_{n}$ the convergents of $\alpha$,
and let $C_{0}$ be as in \eqref{eq:2.4}. If the convergents of $\al$ satisfy \eqref{eq:10.4},
then for all $\e < a_0q_N^{-\gotn}$, with $a_0$ as in \eqref{eq:10.1},
there is at least one quasi-periodic solution $x(t,\e)= c + X(\oo t, \e)$ to \eqref{eq:1.1}, such that
$X(\pps, \e)$ is analytic in $\pps$ in the strip $\Sigma_{\xi'}$, with $\xi'<\xi/4$, is continuous in $\e$,
and goes to zero as $\e \to 0$. 
\end{thm} 

In terms of the quantity $\e_n(\al)$ defined in \eqref{eq:2.9}, the condition \eqref{eq:10.4} reads
\begin{equation} 
\e_n(\al) = \frac{1}{q_n} \log q_{n+1} \le \frac{1}{q_n} \log b_0 + \frac{C_0}{\gotn} .
\end{equation}
In other words, for Theorem \ref{thm:3} to apply one needs $\e_{n}(\al) < C_0/\gotn$. Such a condition is automatically
satisfied not only if $\oo=(1,\al)$ is a Bryuno vector but also if $\e_k(\al)$ goes $0$
(the condition
{\color{black}
usually assumed in KAM theory to ensure that the homological equation be solvable}).


\end{document}